\documentclass[twoside,12pt,a4paper]{article} %
\usepackage{authblk}
\usepackage{float}
\usepackage{amsmath,amssymb,amsfonts}
\usepackage[standard]{ntheorem}
\usepackage{sectsty}
\usepackage{graphicx}
\usepackage{xcolor}

\usepackage{caption}
\usepackage{subcaption}

\usepackage[nocompress]{cite}

\allsectionsfont{\sffamily}
\newcommand{\R}{\ensuremath{\mathbb{R}}}
\newcommand{\scp}[2]{\langle #1, #2 \rangle}

\def\R{{\mathbb R}}

\DeclareMathOperator{\diverg}{div}

\DeclareMathOperator{\Id}{Id}

\DeclareMathOperator{\interior}{int}

\DeclareMathOperator{\sign}{sign}

\DeclareMathOperator{\supp}{supp}
\DeclareMathOperator{\tr}{tr}

\newcommand{\Barg}[1]{\Bigl(#1\Bigr)}
\newcommand{\nset}[1]{\{#1\}}

\newcommand{\nbar}[1]{|#1|}

\newcommand{\bfm}{\mathbf}

\newcommand{\bfeps}{\boldsymbol{\varepsilon}}
\newcommand{\bfsigma}{\boldsymbol{\sigma}}

\newcommand{\be}{\begin{equation}}
\newcommand{\ee}{\end{equation}}
\newcommand{\dev}{\mathrm{dev}}

\newcommand{\D}{\displaystyle}

\DeclareOldFontCommand{\sc}{\normalfont\scshape}{\@nomath\sc}
\numberwithin{equation}{section}

\def\psfiles{.}

\begin{document}

\title{Analytical and numerical relaxation results for models in soil mechanics}
\author[1]{Florian Behr}
\author[1]{Georg Dolzmann}
\author[2]{Klaus Hackl}
\author[2]{Ghina Jezdan}

\affil[1]{\small Faculty of Mathematics; University of Regensburg, 93053 Regensburg, Germany}

\affil[2]{\small Institute of Mechanics of Materials; Ruhr University Bochum, 44780 Bochum, Germany}

\maketitle

\begin{abstract}
A variational model of pressure-dependent plasticity employing a time-incremental setting is introduced. A novel formulation of the dissipation potential allows one to construct the condensed energy in a variationally consistent manner. For a one-dimensional model problem, an explicit expression for the quasiconvex envelope can be found which turns out to be essentially independent of the original pressure-dependent yield surface. The model problem can be extended to higher dimensions in an empirical manner. Numerical simulation exhibit well-posed behavior showing mesh-independent results.
\end{abstract}

\section{Introduction}

Classical models for soils and granular materials are formulated in the framework of linearized plasticity based on pressure-dependent flow rules. While this approach works very well in engineering applications, it cannot explain complex structures seen in experiments conducted in granular media like clay, silt or sand, see in particular \cite[Figure~12]{WOLF20031229}. 
The aim of this paper is both the investigation and the modeling of these complex structures via relaxation methods, a well-established approach based on variational analysis. In a first step, the analysis will be based on a time discretization scheme
which involves a variational formulation for the problems at given finite time-increments. The corresponding variational 
problems will be studied concerning their stability properties and the main features of minimizing microstructures, which are closely related to the complexity of the experimental data. 
Moreover, the dependence of the main features of the minimizing structures on the given yield functions will be 
investigated. The goal is to employ the variational approach in order to describe the essential model features with as few parameters as possible, thus 
establishing this type of models for plasticity in soil mechanics.

Starting with \cite{BallJames1987,BallJames1992,ChipotKinderlehrer1988} it was 
realized that experimentally observed microstructures may be explained based on energy minimization. 
This observation made it possible to use variational methods in order to understand complex 
material behavior. For example, the shape memory effect is related to an accommodation 
of macroscopic deformations on a microscopic scale through complex patterns. 
These patterns are determined through mathematical compatibility conditions 
which require in the nonlinear theory that two adjacent elastic strains $\mathbf{A}$ and $\mathbf{B}$
be rank-one connected in the sense that $\mathbf{A} - \mathbf{B} = \mathbf{a}\otimes \mathbf{n}$. 
The vector $\mathbf{n}$ provides geometric information since 
it represents the orientation of planes that can connect the two states $\mathbf{A}$ and $\mathbf{B}$.
Minimization of the original energy with respect to such compatible microstructures leads to the relaxed energy. This approach allows one on the one hand to formulate well-posed boundary-value problems that lend themselves to numerical solutions. The minimizing microstructures, on the other hand, will provide information on the expected deformation patterns. 
Applications of this approach specifically to elastoplastic materials can be found in \cite{kochmann2011evolution,hackl2012}.

Since the explicit construction of microstructures requires to determine rank-one connections, 
which is algebraically a challenging problem, it is not surprising that there are only a few 
examples with applications to elastic materials for which a closed form of $W^{qc}$ has been found
\cite{CarstensenPlechac1997,DeSimoneDolzmannARMA2002,Silhavy1999,Silhavy2007,ContiDolzmann2014CubicTetragonal,khan2018,khan2021}.
In the linearized setting, the compatibility condition concerns symmetric stains
$\mathbf{A}$ and $\mathbf{B}$ whose difference is symmetrized rank-one convex, i.e., 
$\mathbf{A}-\mathbf{B} = \mathbf{a}\odot \mathbf{b}$ with 
$\mathbf{a}\odot \mathbf{b} = \frac{1}{2}(\mathbf{a}\otimes \mathbf{b} + \mathbf{b}\otimes \mathbf{a})$.
Examples for relaxation results in this context can be found in 
\cite{KohnStrang1986,LurieChervaev1988,KohnCMAT1991,PipkinQuartJMechApplMath1991}. Numerical relaxation schemes are employed in \cite{bartels2004}.
A combination of an analytical approximation of
the relaxed energy {density} combined with a finite element simulation can be
found in~\cite{CarstensenContiOrlando2008,MieheLambrecht2003,MieheLambrechtGuerses04}. A novel algorithmic approach for the computation of the relaxed energy {density}
with applications to problems in phase transformations and models in plasticity 
was recently proposed in \cite{ContiDolzmannJMPS2018,ContiDolzmannJOTA2018,ContiDolzmannJOTAZwei2021}, see also 
\cite{BartelsM2AN2004,BartelsSIAM2005}. 

The paper is organized as follows: In Sec.~\ref{sec2} standard models of pressure dependent plasticity are reviewed. In Sec.~\ref{sec3} we bring these models into a variational setting suitable for treatment within relaxation theory. In Sec.~\ref{sec4} this setting is specialized in order to obtain a one-dimensional model problem for which analytical results can be obtained. Sec.~\ref{sec5} contains the main results of this paper from the mathematical point of view. An explicit expression for the relaxed energy is derived in the one-dimensional case, i.e., the quasiconvex envelope of the energy density, which, in the given situation with one spatial variable, is equal to the convex envelope. In Sec.~\ref{sec6} the one-dimensional relaxed energy is, in a heuristic manner,  extended to three dimensions and in Sec.~\ref{sec7} numerical results are presented for both models, first for the relaxed energy derived for the one-dimensional problem, then for an extension of this energy to three-dimensions, for which no rigorous mathematical results are available yet. More extensive results will be given in \cite{jezdan2023}. We will close with conclusion and outlook in Sec.~\ref{sec8}.

\section{Pressure-dependent plasticity}
\label{sec2}

The most important class of materials exhibiting pressure-dependent plastic behavior are soils and granular media. They are usually heterogeneous mixtures of fluids (usually air and water) and particles (usually clay, silt, sand, and gravel) that have little to no cementation \cite{mitchell2005fundamentals}. The shear strength is provided by friction and interlocking of the particles.  

The understanding of deformations of solid materials is of key importance in applications. In contrast to 
elastic materials, which return to their original configuration upon unloading, elasto-plastic materials 
may undergo 
a permanent change of their configuration if large shear strains occur in the material; in contrast, 
it is frequently assumed that they may 
support arbitrary hydrostatic pressures. Usually this material behaviour is modeled in a linearized setting via a convex set $\mathbb{K}\subset \R^{n\times n}_{\mathrm{sym}}$
which describes the set of admissible stresses in the material which lead to elastic deformations. Only for stresses 
on $\partial \mathbb{K}$, permanent plastic deformation occurs and this deformation is formulated via a flow rule which
determines the evolution of the plastic part of the total elastoplastic deformation. Consequently, the set $\mathbb{K}$
is an infinite cylinder 
\begin{align}
\mathbb{K}=\bigcup\limits_{\pi \in \R} \pi\mathbf{Id} \oplus K=\bigcup\limits_{\pi=\pi \in \R}^{\pi_\mathrm{max}}\{ \pi \mathbf{Id} + \mathbf{E},  \mathbf{E}\in K \}\,,
\end{align}
where $K$ is a compact subset in $\R^{n\times n}_{\mathrm{D}}$ and
\begin{align}
\R^{n\times n}_{\mathrm{D}} = \{ \mathbf{E}\in \R^{n\times n}\colon \mathbf{E} = \mathbf{E}^T, \, \tr(\mathbf{E})=0\}
\end{align}
is the set of all symmetric and trace free matrices. Here and in the following we will use boldface symbols for 
vector-valued or matrix-valued objects.

Thus the basic problem in the theory of linear elastoplasticity can be formulated as follows. Given a material body described 
in its reference configuration by a domain $\Omega\subset\R^n$, find a deformation $\mathbf{y}:\Omega\to \R^n$ and 
a decomposition of the symmetrized displacement gradient $\boldsymbol{\epsilon}$ defined by 
\begin{align}
\mathbf{y}(\mathbf{x}) = \mathbf{id}(\mathbf{x}) + \mathbf{u}(\mathbf{x})\,,\quad 
\boldsymbol{\epsilon}(\mathbf{x}) = \frac{1}{2}( D\mathbf{u}(\mathbf{x}) + D\mathbf{u}(\mathbf{x})^T)
\end{align}
into an elastic and a plastic part, $\boldsymbol{\epsilon}=\boldsymbol{\epsilon}_e + \boldsymbol{\epsilon}_p$, 
$\boldsymbol{\epsilon}_p \in \R^{n\times n}_{\mathrm{D}}$, such that the following system of equations holds:
\begin{itemize}
	\item [(i)] constitutive equation $\boldsymbol{\sigma} = \mathbb{C}\boldsymbol{\epsilon}_e$,
	\item [(ii)] balance of momentum $\mathbf{u}_{tt}-\diverg \boldsymbol{\sigma} = \mathbf{f}$ 
	or the equilibrium condition $-\diverg\boldsymbol{\sigma} = \mathbf{f}$,
	\item [(iii)] stress constraint $\boldsymbol{\sigma} \in \mathbb{K}$,
	\item [(iv)] associated flow rule $(\boldsymbol{\xi} -\boldsymbol{\sigma}):\dot{\boldsymbol{\epsilon}}_p \leq 0$
	for every $\boldsymbol{\xi}\in \mathbb{K}$.
\end{itemize}

It is important to note that the assumption (iv) is not realistic for materials in which frictional effects dominate as in 
clay, sands, gravel or other granular materials. Here pressure-dependence needs to be taken into account 
and this leads to a flow rule that cannot be modeled by an infinite cylinder; instead one is led to 
consider cones in the space of principle stresses, classical models being the Drucker-Prager and 
Mohr-Coulomb criteria \cite{Lubliner2008plasticity}. To be specific, let $\pi=- 1/3 \, \tr\bfsigma$ denote the pressure, 
then the convex elastic domain is given as
\begin{align}
\mathbb{K}=\bigcup\limits_{\pi=\pi_\mathrm{min}}^{\pi_\mathrm{max}} \pi\mathbf{Id} \oplus K_\pi=\bigcup\limits_{\pi=\pi_\mathrm{min}}^{\pi_\mathrm{max}}\{ \pi \mathbf{Id} + \mathbf{E},  \mathbf{E}\in K_\pi \}\,,
\end{align}
where $K_\pi$ 
is a compact subset in $\R^{n\times n}_{\mathrm{D}}$ and $-\infty\leq\pi_\mathrm{min}<\pi_\mathrm{max}\leq\infty$.
In order to comply with experimental observations, plastic deformation should still be volume preserving. For this purpose, the flow rule (iv) is then replaced by the non-associated flow rule and the full system of equations 
in its quasistatic form is given by
\begin{itemize}
	\item [(i)] constitutive equation $\boldsymbol{\sigma} = \mathbb{C}\boldsymbol{\epsilon}_e$,
	\item [(ii)] the equilibrium condition $-\diverg\boldsymbol{\sigma} = \mathbf{f}$,
	\item [(iii)] stress constraint $\boldsymbol{\sigma} \in \mathbb{K}$,
	\item [(iv')] non-associated flow rule $(\boldsymbol{\xi} -\boldsymbol{\sigma}):\dot{\boldsymbol{\epsilon}}_p \leq 0$
	for every $\xi\in \mathbb{K}_\pi$.
\end{itemize}
Note that $\dot{\boldsymbol{\epsilon}}_p$ is now normal to the boundary of $\mathbb{K}_\pi$ as subset of $\R^{n\times n}_{\mathrm{D}}$
but not normal to $\mathbb{K}$ as subset of $\R^{n\times n}_{\mathrm{sym}}$. 
This observation lies at the heart of the problems concerning a variational formulation to be discussed in the following.

\section{Variational model}
\label{sec3}

In order to be able to apply relaxation theory, we have to reformulate the setting given in the previous section. For this purpose, we will employ a variational approach for the description of inelastic processes, see \cite{moreau1968,carstensen2002non-convex,miehe2002homogenization,ortiz1999nonconvex,HACKL1997667,hackl2008relation}.
Let us consider a physical system described by (sets of) external, i.e., controllable, state variables, in our case given by the total strain $\bfeps$, and internal state variables $\mathbf{z}$. Here $\mathbf{z}=\{\bfeps_\mathrm{p},p\}$, where $\bfeps_\mathrm{p}$ denotes plastic strain and $p\geq 0$ a hardening variable.
We assume that the system behavior may be defined using only two scalar potentials: free energy $\psi(\bfeps,\mathbf{z})$ and dissipation potential $\Delta(\bfeps,\mathbf{z},\dot{\mathbf{z}})$.

The evolution of the internal variables is described by 
\begin{equation}
\label{eq1}
{\mathbf 0} \in \frac{\partial \psi}{\partial {\mathbf{z}}}+\frac{\partial \Delta}{\partial \dot{\mathbf{z}}}.
\end{equation}
Note, that Eq.~\eqref{eq1} may be written as a stationarity condition of the minimization problem
\begin{equation}
\label{eq2}
\dot{\mathbf{z}} = \arg\inf_{\dot{\mathbf{z}}}\left\{{\dot{\psi} + \Delta}\right\}.
\end{equation}

We assume a free energy of the form
\be
\label{eq3}
\psi =  \frac{K}{2} \, {\tr \bfeps}^2 + \mu \left\|\dev \bfeps  - \bfeps_\mathrm{p} \right\|^2 + \rho(\tr \bfeps) \, p + \frac{\beta}{2} \left\|\bfeps_\mathrm{p} \right\|^2.
\ee
Here, $p$ is an auxiliary internal variable, $K$ the bulk-modulus, $\mu$ the shear-modulus, and $\beta$ a hardening-modulus. The function $\rho(\tr \bfeps)$ governs pressure-dependence. For constant $\rho(\tr \bfeps)$, Eq.~\eqref{eq3} corresponds to the free energy of an isotropic elastoplastic material with kinematic hardening. In this work, we will assume $\rho$ to be a concave function inside an interval $[\xi_\mathrm{min},\xi_\mathrm{max}]$ and zero outside of it. The function $\rho(\tr \bfeps)$ for a typical model, e.g. capped Drucker-Prager, see \cite{mitchell2005fundamentals}, can be seen in Fig.~\ref{fig-rho(xi)}. The stress is given as
\begin{figure}
\centering
\includegraphics[scale=1.2]{\psfiles/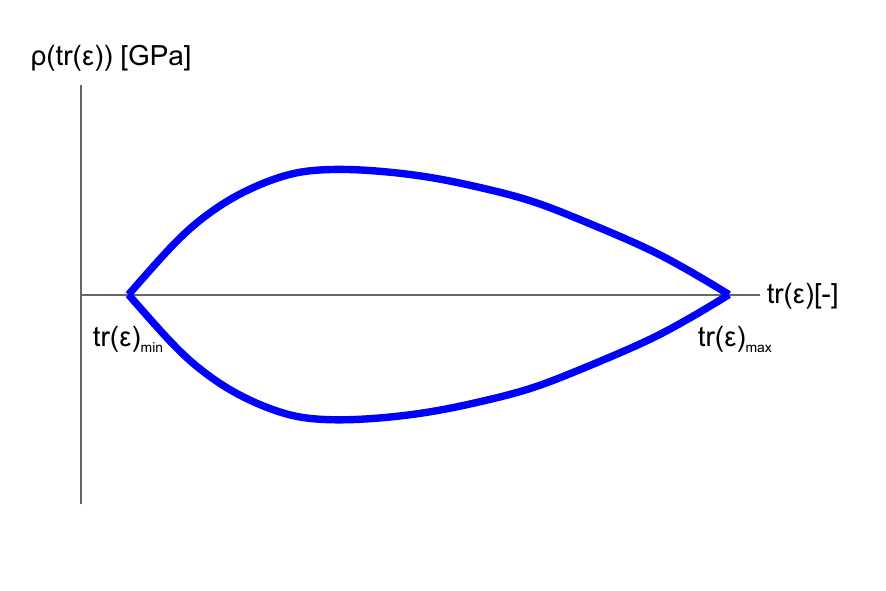} 
\caption{typical pressure-dependent yield surface. \label{fig-rho(xi)}}
\end{figure}
\begin{equation}
 \label{eq3a}
\bfsigma = \frac{\partial \psi}{\partial \bfeps} = (K \, \tr\bfeps +\rho^\prime(\tr\bfeps) \, p) \mathbf{I} + 2\mu \left(\dev \bfeps  - \bfeps_\mathrm{p} \right).
\end{equation}
Note, that the pressure is given by the relation $\pi = - K \tr \bfeps - \rho^\prime(\tr\bfeps) p$. Here, the second term may be interpreted as a consolidation pressure due to an inelastic volume change governed by the parameter $p$. A variationally more consistent approach would be to replace  Eq.~\eqref{eq3} by an energy of the form
\be
\label{eq3b}
\psi_\mathrm{consol} =  \frac{K}{2} \, (\tr \bfeps - \theta_\mathrm{p})^2 + \mu \left\|\dev \bfeps  - \bfeps_\mathrm{p} \right\|^2 + \rho(\theta_\mathrm{p}) \, p + \frac{\beta}{2} \left\|\bfeps_\mathrm{p} \right\|^2,
\ee	
where $\theta_\mathrm{p}$ now represents the inelastic volume change explicitly. However, we would like to keep our formulation simple within the present work and plan to discuss this formulation in a subsequent paper.

The dissipation potential will be assumed as
\be
\label{eq4}
\Delta(\dot{\bfeps}_\mathrm{p},\dot p) = 
\begin{cases}
0 & \text{for  } \|\dot{\bfeps}_\mathrm{p}\| \leq \dot p \\
\infty & \text{otherwise}
\end{cases}.
\ee
This procedure differs from the standard one, as for example implemented in \cite{carstensen2002non-convex}. In the standard approach, Eq.~\eqref{eq3} is replaced by
\be
\label{eq3c}
\psi_\mathrm{alt} =  \frac{K}{2} \, {\tr \bfeps}^2 + \mu \left\|\dev \bfeps  - \bfeps_\mathrm{p} \right\|^2 + p + \frac{\beta}{2} \left\|\bfeps_\mathrm{p} \right\|^2,
\ee
and Eq.~\eqref{eq4} by
\be
\label{eq4a}
\Delta_\mathrm{alt}(\tr \bfeps,\dot{\bfeps}_\mathrm{p},\dot p) = \rho(\tr\bfeps)
\begin{cases}
	0 & \text{for  } \|\dot{\bfeps}_\mathrm{p}\| \leq \dot p \\
	\infty & \text{otherwise}
\end{cases}.
\ee
However, the formulation in Eq.~\eqref{eq4a} would be problematic with respect to the variational formulation which does not allow the dissipation potential to depend on an external variable. Moreover, it would lead to complications when it comes to integrating it up in order to obtain a dissipation distance, as will be done later in the text. Minimizing trajectories connecting different states will be non-monotone leading to behavior that is questionable from a physical point of view.

Minimization in Eq.~\eqref{eq2} gives, taking into account the expressions in Eqs.~\eqref{eq3} and \eqref{eq4},
\be
\label{eq5}
\dev\bfsigma - \beta \bfeps_\mathrm{p} \in \rho(\tr \bfeps) \sign \dot{\bfeps}_\mathrm{p},
\ee
and
\be
\label{eq6}
\dot p = \|\dot{\bfeps}_\mathrm{p}\|.
\ee
Note, that Eq.~\eqref{eq5} is equivalent to the yield condition
\be
\label{eq7}
\Phi := \|\dev\bfsigma - \beta \bfeps_\mathrm{p}\| - \rho(\tr \bfeps) \leq 0,
\ee
and the flow rule
\be
\label{eq8}
\dot{\bfeps}_\mathrm{p} = \kappa \, (\dev\bfsigma - \beta \bfeps_\mathrm{p}),
\ee
where $\kappa\geq 0$ is a consistency parameter satisfying the Kuhn-Tucker conditions
\be
\label{eq9}
\kappa\geq 0, \qquad \Phi \leq 0, \qquad \kappa \, \Phi =0.
\ee
Note, that Eq.~\eqref{eq6} defines $p$ as equivalent plastic strain.

Having established a variational scheme for pressure-dependent plasticity now, we will employ a time-incremental approach as proposed in \cite{carstensen2002non-convex,ortiz1999nonconvex} in order to investigate the evolution of material microstructures.
For this purpose, we consider a time-increment $\Delta t=[t_n,t_{n+1}]$. Given the external forces and boundary conditions at time $t_{n+1}$ and the values of the internal variables $\mathbf{z}_n$ at time $t_n$, we seek the values $\bfeps_{n+1}$ and $\mathbf{z}_{n+1}$ at time $t_{n+1}$.

A time-incremental version of the dissipation-potential is given by the dissipation-distance
The extremal principle Eq.~\eqref{eq2} can now be replaced by
\be
\label{eq11}
\{\bfeps_{\mathrm{p},n+1},p_{n+1}\} = \arg\inf_{\bfeps_\mathrm{p},p}\left\{{\psi(\bfeps,\bfeps_\mathrm{p},p) + D(\bfeps_{\mathrm{p},n},p_n,\bfeps_\mathrm{p}},p)\right\}.
\ee
Minimization with respect to $p$ gives immediately
\be
\label{eq14}
p - p_n = \|\bfeps_\mathrm{p} - \bfeps_{\mathrm{p},n}\|.
\ee
Then, substitution of Eq.~\eqref{eq14} into Eq.~\eqref{eq11} and minimization with respect to $\bfeps_\mathrm{p}$ yields
\be
\label{eq13}
2\mu (\dev\bfeps - \bfeps_\mathrm{p}) - \beta \bfeps_\mathrm{p} \in \rho(\tr \bfeps) \sign (\bfeps_\mathrm{p} - \bfeps_{\mathrm{p},n}).
\ee
Solving for $\bfeps_\mathrm{p}$ in Eq.~\eqref{eq13}yields
\be
\label{eq15}
\bfeps_\mathrm{p} = \bfeps_{\mathrm{p},n} + \frac{1}{2\mu+\beta} \, \left[2\mu \|\dev\bfeps - \bfeps_{\mathrm{p},n}\| - \rho(\tr \bfeps)  \right]_+ \sign (\dev\bfeps - \bfeps_{\mathrm{p},n}),
\ee
where $\left[x\right]_+=\max\{0,x\}$. This allows to introduce the condensed energy, see \cite{carstensen2002non-convex,ortiz1999nonconvex}, via
\be
\label{eq16}
\psi_\mathrm{cond}(\bfeps,\bfeps_{\mathrm{p},n},p_n) = \inf_{\bfeps_\mathrm{p},p}\left\{{\psi(\bfeps,\bfeps_\mathrm{p},p) + D(\bfeps_{\mathrm{p},n},p_n,\bfeps_\mathrm{p}},p)\right\}.
\ee
Substitution of Eqs.~\eqref{eq14} and \eqref{eq15} into Eq.~\eqref{eq16} gives, up to an irrelevant constant term of the form $\frac{\beta}{2} \left\|\bfeps_{\mathrm{p},n} \right\|^2$,
\begin{multline}
\label{eq17}
\psi_\mathrm{cond}(\bfeps,\bfeps_{\mathrm{p},n},p_n) = 
\frac{K}{2} \, {\tr \bfeps}^2 + \mu \left\|\dev \bfeps  - \bfeps_{\mathrm{p},n} \right\|^2 + \rho(\tr \bfeps) p_n\\
- \frac{1}{2} \, \frac{1}{2\mu+\beta} \, \left[2\mu \|\dev\bfeps - \bfeps_{\mathrm{p},n}\| - \rho(\tr \bfeps)  \right]_+^2.
\end{multline}

The condensed energy can now be considered as formally elastic and dependent on the parameters $\bfeps_{\mathrm{p},n},p_n$, governing the behavior during the time-increment $\Delta t=[t_n,t_{n+1}]$. At the end of the time-step the internal variables can be updated using Eqs.~\eqref{eq14} and \eqref{eq15}. This way, the inelastic behavior can be approximated by solving a sequence of elastic problems.

\section{A one-dimensional model problem}
\label{sec4}

It is well-known that quasiconvex envelopes can be calculated analytically only in rare special cases. In order to circumvent this difficulty, we restrict ourselves to a one-dimensional problem for which the quasiconvex envelope coincides with the convex envelope. We consider a state of specific deformation corresponding to a thin sheet of material under uniform tension or compression and shear. We assume the sheet to be supported by stiff fibers in a way that prevents bending deformations. The non-vanishing displacements and strains are in Cartesian coordinates
\be
\label{eq18}
u_x = u(x), \qquad u_y = v(x)+y \, u^\prime(x), \qquad u_z = z \, u^\prime(x)
\ee
\be
\label{eq19}
\varepsilon_{xx} = \varepsilon_{yy} = \varepsilon_{zz} = u^\prime, \qquad \varepsilon_{yx} = \varepsilon_{xy} = \frac{1}{2} \, v^\prime,
\ee
where we neglected terms of order $u^{\prime\prime}$.
For the plastic strains, we assume that only the components $\varepsilon_{\mathrm{p}} := \varepsilon_{\mathrm{p}yx} = \varepsilon_{\mathrm{p}xy}$ are different from zero. The free energy assumes the form
\be
\label{eq20}
\psi = \frac{K}{2} \, \varepsilon_{xx}^2 + \mu \left( \left( \varepsilon_{yx} - \varepsilon_{\mathrm{p}yx} \right)^2 + \left( \varepsilon_{xy} - \varepsilon_{\mathrm{p}xy} \right)^2 \right) 
+ \rho(\varepsilon_{xx}) \, p + \beta \, \varepsilon_{\mathrm{p}xy}^2.
\ee
Substitution of Eq.~\eqref{eq19} into Eq.~\eqref{eq20} gives
\be
\label{eq21}
\psi = \frac{K}{2} \left(u^\prime \right)^2 + 2 \, \mu \left(  \frac{1}{2} v^\prime - \varepsilon_\mathrm{p} \right)^2
+ \rho\left(u^\prime\right) \, p + \beta \, \varepsilon_\mathrm{p}^2.
\ee
As before, minimization with respect to $\varepsilon_{\mathrm{p}}$ and $p$ once again leads to the condensed energy
\begin{multline}
\label{eq25}
\psi_\mathrm{cond}(u^\prime,v^\prime,\varepsilon_{\mathrm{p},n},p_n) = \frac{K}{2} \, \left(u^\prime \right)^2 + 2 \, \mu \, \left( \frac{1}{2} \, v^\prime -\varepsilon_{\mathrm{p},n} \right)^2 + \rho\left(u^\prime \right) p_n\\
- \frac{1}{2} \, \frac{1}{2\mu+\beta} \, \left[2\mu \left|\frac{1}{2} \, v^\prime - \varepsilon_{\mathrm{p},n}\right| - \rho\left(u^\prime \right) \right]_+^2.
\end{multline}
For the remainder of the paper, we will focus on the case $p_n=0$. This will simplify the mathematical treatment significantly and will allow to keep the derivations stringent and instructive. In mechanical terms this means, that we will describe the initiation of microstructures rather than their evolution. However, the general case will remain subject of future investigations.

It is convenient for the mathematical treatment following in the subsequent section to introduce a dimensionless formulation. For this purpose, we define modified variables and quantities by
\be
\label{eq26}
y_1 = \sqrt{\frac{K}{2\mu}} \, u^\prime, \quad y_2 = \frac{1}{2} \, v^\prime, \quad r(y_1) = \frac{1}{2\mu} \, \rho\left(u^\prime \right), \quad b = \frac{\beta}{2\mu},
\ee
giving, using $p_n=0$,
\begin{multline}
\label{eq27}
\psi_\mathrm{cond}(y_1,y_2,\varepsilon_{\mathrm{p},n},p_n) =
2\mu \, \left[ \frac{1}{2} \, y_1^2 + \frac{1}{2} \left( y_2 -\varepsilon_{\mathrm{p},n} \right)^2
- \frac{1}{2} \, \frac{1}{b+1} \, \left[ \, |y_2  - \varepsilon_{\mathrm{p},n}| - r(y_1) \right]_+^2  \right].
\end{multline}

Finally we may drop the multiplicative constant $2\mu$ and write $z_n =  \varepsilon_{\mathrm{p},n}$. Since $\psi_\mathrm{cond}$ is the free energy density in the variational functional that needs to be minimized in the first time step, it is crucial to understand the convexity properties of $\psi_\mathrm{cond}$ in order to employ the direct method in the calculus of variations. It turns out that the concavity of $r$ leads to a lack of convexity of $\psi_\mathrm{cond}$ and the approach via relaxation~\cite{Dacorogna1989} requires one to characterize the quasiconvex envelope of $\psi_\mathrm{cond}$, which, in the scalar case, reduces to the convex envelope. This task is accomplished in the next section. 

\section{Characterization of the relaxed energy for the one-dimensional model}
\label{sec5}

For any function $f: \R^2\to \R$ we denote the convex envelope of $f$ with $f_\mathrm{c}$ and, if $f$ is differentiable at $y\in \R^2$, the linear Taylor polynomial of $f$ about the point $y$ with $T_y f$. Fix $b>0$ and $y_{\mathrm{min}}, y_{\mathrm{max}}\in \R$ with $y_{\mathrm{min}}< y_{\mathrm{max}}$. Define the set of admissible dissipation functions  by 
\begin{align}
 \mathcal{R} = \nset{ r\in C(\R, [0,\infty)),\, \supp(r)= [y_{\mathrm{min}},y_{\mathrm{max}}], r|_{[y_{\mathrm{min}},y_{\mathrm{max}}]}\text{ concave} }\,.
\end{align}
Fix $y_{\mathrm{mid}}:= (y_{\mathrm{min}}+y_{\mathrm{max}})/2$ and define the piecewise affine function $r_0\in \mathcal{R}$ by
\begin{align}
r_0(y_1)= \sqrt{b}\cdot \left( \frac{y_{\mathrm{max}}-y_{\mathrm{min}}}{2}- \nbar{ y_1 - y_{\mathrm{mid}} } \right)\cdot \chi_{[y_{\mathrm{min}},y_{\mathrm{max}}]}(y_1)\,.
\end{align}
The condensed energy for some $r\in \mathcal{R}$ is then given by 
\begin{align}\label{FormulasCondensedEnergy}
f^{(r)}_{z_n}(y) = \frac{1}{2}(y_1^2 + (y_2- z_n)^2) - \frac{(|y_2- z_n|- r(y_1))_+^2}{2(b+1)}\,.
\end{align}
Since $y_2$ appears only shifted by $z_n$ it is sufficient to calculate the convex envelope $f_\mathrm{c}^{(r)}$ of $f^{(r)}:= f^{(r)}_0$ and obtain the convex envelope of $f^{(r)}_{z_n}(y_1,y_2)$ by $f_\mathrm{c}^{(r)}(y_1, y_2- z_n)$.\\
If the dissipation function is apparent from the context, we simply write $f$ instead of $f^{(r)}$. Define the disjoint sets 
\begin{align}\label{def:regions}
\begin{aligned}
  \widetilde{Y} & = \nset{ y\in \R^2 \colon y_1 \leq y_{\mathrm{min}} \text{ or }y_1 \geq y_{\mathrm{max}} }\,, \\
 Y_0 & =\nset{ y\in \R^2 \colon y_1\in (y_{\mathrm{min}},y_{\mathrm{max}}),\, |y_2|< r(y_1) }\,, \\ 
 Y_\pm & = \nset{ y\in \R^2\colon y_1\in (y_{\mathrm{min}},y_{\mathrm{max}}),\,\pm y_2 \geq r(y_1) }\,,
\end{aligned}
\end{align}
whose union is $\R^2$, see Figure~\ref{fig:DefinitionModelProblemEnergy}.

\begin{figure}[h]
\centering 
\includegraphics[width=.85\textwidth]{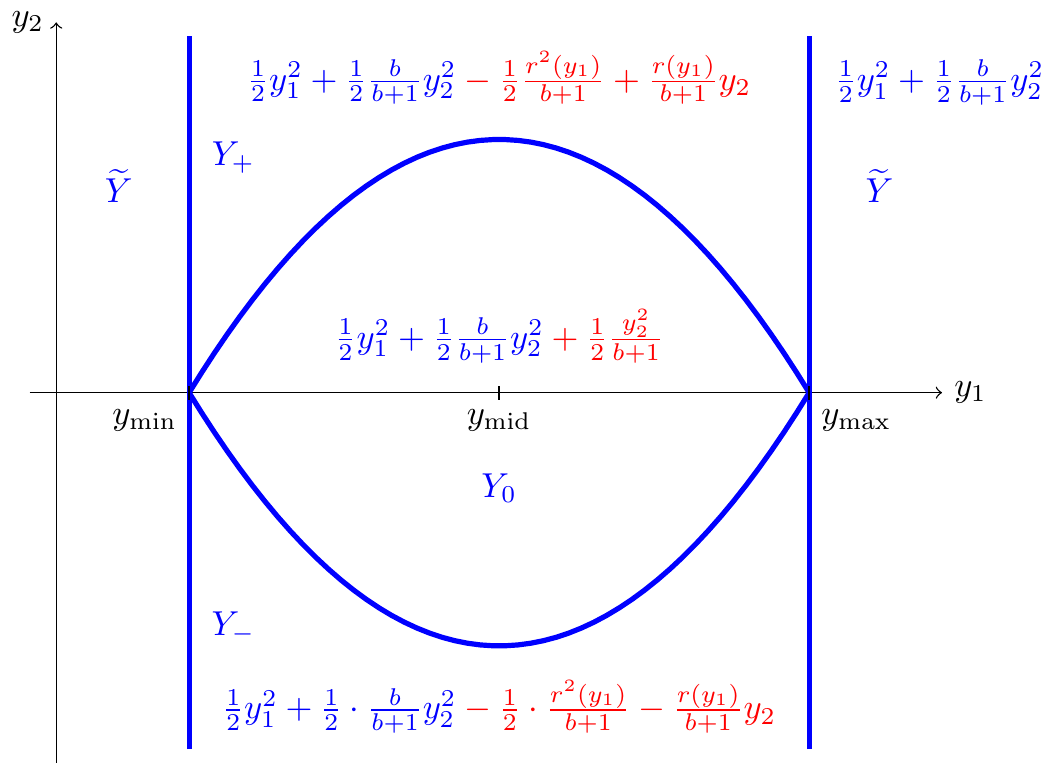}
\caption{Definition of the condensed energy for the model problem. The set $Y_0$ in Eq.~\eqref{def:regions} is enclosed by the blue curves ($\pm r$), the sets $Y_\pm$ correspond to the sets above and below $Y_0$, respectively, and the set $\widetilde{Y}$ is the complement of these three sets in $\R^2$. \label{fig:DefinitionModelProblemEnergy}}
\end{figure}

\begin{lemma}\label{lem:OrderCondensedEnergy}
Consider two dissipation functions $r, \bar{r}\in \mathcal{R}$. If $y_1\in \R$ and $r(y_1)\leq \bar{r}(y_1)$, then for all $y_2\in \R$ the inequality $f^{(r)}(y_1, y_2)\leq f^{(\bar{r})}(y_1, y_2)$ holds. In particular, if $r\leq \bar{r}$, then $f^{(r)}\leq f^{(\bar{r})}$ and $g: \R^2\to \R,\ g(y)= \frac{1}{2}y_1^2 + \frac{1}{2} \frac{b}{b+1} y_2^2$ is a convex lower bound of $f^{(r)}$, which coincides on $\widetilde{Y}$ with $f^{(r)}$.
\end{lemma}

\begin{proof}
For any $y\in \mathbb{R}^2$ with $r(y_1)\leq \bar{r}(y_1)$ it follows that $|y_2|\geq |y_2|- r(y_1)\geq |y_2|- \overline{r}(y_1)$. Therefore we have $y_2^2\geq (|y_2|- r(y_1))_+^2\geq (|y_2|- \overline{r}(y_1))_+^2$ and the inequality
\begin{align}
g(y)\leq f^{(r)}(y_1, y_2)\leq f^{(\overline{r})}(y_1, y_2)
\end{align}
is obtained from Eq.~\eqref{FormulasCondensedEnergy}.
\end{proof}

Since Lemma~\ref{lem:OrderCondensedEnergy} establishes a lower bound on $f_\mathrm{c}$, we may use the results in \cite[Theorem~2.1]{GriewankRabier1990}. For a lower semicontinuous function $f: \R^d\to \R$, which satisfies the growth condition $\lim_{|\delta|\to \infty} f(\delta)/|\delta|= \infty$, the convex envelope of $f$ is a real-valued convex function and for any $y\in \R^d$ the value $f_\mathrm{c}(y)$ is a convex combination of at most $d+1$ function values $f(y_i)$, $1\leq i\leq q\leq d+1$.
Moreover, the subgradient $\partial f_\mathrm{c}(y)$ is nonempty and for any $v\in \partial f(y)$
\begin{align} \label{ConvEnv1}
f(y_i)- \scp{v}{y_i}= f(y)- \scp{v}{y}\,, \qquad i\in \{1,...,q\}
\end{align}
and
\begin{align} \label{ConvEnv2}
\forall \bar{y}\in \R^d: f(\bar{y})- \scp{v}{\bar{y}}\geq f(y_i)- \scp{v}{y_i}\,, \qquad i\in \{1,...,q\}\,.
\end{align}
Since Eq.~\eqref{ConvEnv1} implies $f(y_i)- \scp{v}{y_i}= f(y_j)- \scp{v}{y_j}$ for all $i,j\in \{1,...,q\}$ and Eq.~\eqref{ConvEnv2} is equivalent to $v\in \partial f(y_i)$, a natural approach to calculate the convex envelope of a continuous function $f: \R^d\to \R$ is therefore to find pairwise distinct points $y_1,...,y_q\in \R^d$ ($q\leq d+1$), which satisfy the subgradient condition
\begin{align} \label{SubgradientCondition} \tag{SG}
\exists v\in \bigcap\limits_{i=1}^q \partial f(y_i): \forall i,j\in \{1,...,q\}: f(y_i)- \scp{v}{y_i}= f(y_j)- \scp{v}{y_j}
\end{align}
and define for all $y\in \text{Conv}(\{y_1,...,y_q\})$, the convex hull of the points $y_1,\ldots,y_q$, the convex envelope of $f$ by 
\begin{align} \label{ConstructionConvexEnvelope}
y= \sum\limits_{i=1}^q \lambda_i y_i\,,\quad f_\mathrm{c}(y)= \sum\limits_{i=1}^q \lambda_i f(y_i)\,.
\end{align}
Finally, if for some direction $w\in \R^d$ and some $y\in \R^d$ the directional derivative $D_w f(y)$ exists, then $v\in \partial f(y)$ with $|v|=1$ implies $\scp{v}{w}= D_w f(y)$. Furthermore, if $f$ is differentiable in $y$, then $v\in \partial f(y)$ implies $v= \nabla f(y)$. This leads to the observation that, if $y_1,...,y_q\in \R^d$ is a solution of Eq.~\eqref{SubgradientCondition} and if $f$ is differentiable in some $y_i$, then for all $y\in \text{Conv}(\{y_1,...,y_q\})$ the convex envelope of $f$ is given by
\begin{align} \label{ConstructionConvexEnvelopeDifferentiable}
f_\mathrm{c}(y)= (T_{y_i}f)(y)= f(y_i)+ \scp{\nabla f(y_i)}{y- y_i}\,.
\end{align}
The calculation of the subgradient involves an evaluation of the values of $f$ on the whole space, and the solution of Eq.~\eqref{SubgradientCondition} is in general challenging.

Based on these observations, for the construction of the convex envelope of the condensed energy $f^{(r)}$ we first derive pairs and triples of points $y_1,...,y_q\in \R^2$ ($q= 2,3$), which have equal partial derivatives in the $y_2$ direction in which $f^{(r)}$ is differentiable. Then we derive a candidate for the direction $v\in \R^d$ satisfying the equation in Eq.~\eqref{SubgradientCondition}. However, at this point we do not know whether this $v$ in Eq.~\eqref{SubgradientCondition} is indeed a subgradient of $f$ in all the points $y_1, ..., y_q$. For those pairs and triples for which we expect this to be true, we do not show explicitly that the first part in Eq.~\eqref{SubgradientCondition} is satisfied, rather we apply the construction described in Eqs.~\eqref{ConstructionConvexEnvelope} and \eqref{ConstructionConvexEnvelopeDifferentiable} and prove in Theorem~\ref{thm:RelaxationSmallb} that this construction leads to the convex envelope. It is a remarkable result of the subsequent analysis that the explicit formulas for the relaxed energy show a very different behavior with respect to the dependence on the dissipation function $r$. In fact, for small values of $b$ relative to $r$, i.e., for
\begin{align} \label{ConditionSmallb}
 \sqrt{b} \cdot \frac{y_{\mathrm{max}}-y_{\mathrm{min}}}{2} \leq r\Barg{ \frac{y_{\mathrm{min}}+y_{\mathrm{max}}}{2} }\,,
\end{align}
which is the case of relevance in soil mechanics, the relaxed energy is in fact independent of $r$. Therefore we restrict the analysis in this paper to the case of small $b$, the general case will be treated in~\cite{Behr2023}. In fact, this assumption implies $r_0\leq r$ by concavity of $r$. In the following we consider $r$ as fixed, write $f= f^{(r)}$ and assume without loss of generality that $r_0<r$ on $(y_{\mathrm{min}},y_{\mathrm{max}})$.

In order to implement the strategy just outlined, we construct a candidate $f_\mathrm{a}$ for $f_\mathrm{c}$. In a first step, we search for points with common supporting planes in $Y_0$ and $\widetilde{Y}$, see Figure~\ref{fig:DefinitionModelProblemEnergy}. Assume thus that $(y,\widetilde{y})\in Y_0\times \widetilde{Y}$ satisfy Eq.~\eqref{SubgradientCondition}. Since $f$ is differentiable in $y$, $v= \nabla f(y)$ is the only possible choice for $v$ in Eq.~\eqref{SubgradientCondition}. If $\widetilde{y}\in \interior(\widetilde{Y})$, then $f$ is differentiable in $\widetilde{y}$ and $\nabla f(y)= v\in \partial f(\widetilde{y})\subset \{\nabla f(\widetilde{y})\}$ implies $\widetilde{y}_1= \partial_1 f(\widetilde{y})= \partial_1 f(y)= y_1$, a contradiction. Consequently we assume that $\widetilde{y}\in \partial\widetilde{Y}$, that is, $\widetilde{y}_1\in \{y_{\mathrm{min}}, y_{\mathrm{max}}\}$. The existence of $\partial_2 f(\widetilde{y})$ implies $\partial_2 f(\widetilde{y})= v_2= \partial_2 f(y)$ and equivalently
\begin{align}
\frac{b}{b+1}\widetilde{y}_2 = y_2 \quad \Leftrightarrow\quad 
 \widetilde{y}_2 - y_2 = \frac{1}{b}y_2\,.
\end{align}
The equation in Eq.~\eqref{SubgradientCondition} is equivalent to $f(\widetilde{y}) = (T_y f)(\widetilde{y})$. By definition of $f$ and since $f|_{Y_0}$ is a quadratic function, these conditions lead to
\begin{align}
\frac{1}{2}\widetilde{y}_1^2 + \frac{1}{2} \frac{b}{b+1} \widetilde{y}_2^2
= f(\widetilde{y})= (T_y f)(\widetilde{y})
= \frac{1}{2}\widetilde{y}_1^2 + \frac{1}{2} \widetilde{y}_2^2 - \frac{1}{2}\nabla^2f(y)[\widetilde{y}-y, \widetilde{y}-y]\,.
\end{align}
In view of $\nabla^2 f(y) = \Id$, the algebraic condition can be simplified to 
\begin{align}
 (\widetilde{y}_1 - y_1)^2 + (\widetilde{y}_2 - y_2)^2 = \frac{1}{b+1}\widetilde{y}_2^2\,.
\end{align}
One defines $s= |\widetilde{y}_1-y_1|$ and substitutes the relation between $\widetilde{y}_2$ and $y_2$ to obtain 
\begin{align}
 y_2 = \pm\sqrt{b} s \quad \Leftrightarrow\quad 
 y_2 = \pm \sqrt{b} s\,.
\end{align}
Since $\widetilde{y}_1\in \{y_{\mathrm{min}}, y_{\mathrm{max}}\}$, one obtains four one-parameter families of pairs of points,
\begin{align}
\alpha_{\mathrm{min}}^\pm(s) = \left(\begin{array}{cc} y_{\mathrm{min}}+ s\\ \pm \sqrt{b}\cdot s \end{array} \right)&\,,\quad 
\beta_{\mathrm{min}}^\pm(s) = \left(\begin{array}{cc} y_{\mathrm{min}}\\ \pm \frac{b+1}{\sqrt{b}} s \end{array} \right)\,,\\
\alpha_{\mathrm{max}}^\pm(s) = \left(\begin{array}{cc} y_{\mathrm{max}}- s\\ \pm \sqrt{b}\cdot s \end{array} \right)&\,,\quad 
\beta_{\mathrm{max}}^\pm(s) = \left(\begin{array}{cc} y_{\mathrm{max}}\\ \pm \frac{b+1}{\sqrt{b}} s \end{array} \right)\,.
\end{align}
where $s$ can take all values in $(y_{\mathrm{min}},y_{\mathrm{max}})$ with $\alpha_{\mathrm{min}}^\pm(s)\in Y_0$ (or $\alpha_{\mathrm{max}}^\pm(s)\in Y_0$). For the construction of the convex envelope only the values $0< s< (y_{\mathrm{max}}- y_{\mathrm{min}})/2=: s_\ast$ will be relevant, since the curves $\alpha_{\mathrm{min}}^\pm$ and $\alpha_{\mathrm{max}}^\pm$ intersect in the point
\begin{align}
y_\ast^\pm:= \alpha_{\mathrm{min}}^\pm(s_\ast)= \alpha_{\mathrm{max}}^\pm(s_\ast)= (y_{\mathrm{mid}},\pm \sqrt{b}(y_{\mathrm{max}}- y_{\mathrm{min}})/2)\in Y_0\,.
\end{align}
For each $y\in \R^2$, which lies on a connecting straight line of two corresponding points of the curves $\alpha_{\mathrm{min}}^\pm$ and $\beta_{\mathrm{min}}^\pm$, i.e., $y= t\cdot \alpha_{\mathrm{min}}^\pm(s)+ (1-t)\cdot \beta_{\mathrm{min}}^\pm(s)$ with $s\in (0,s_\ast)$ and $t\in [0,1]$, we define $f_\mathrm{a}(y)= t\cdot f(\alpha_{\mathrm{min}}^\pm(s))+ (1-t)\cdot f(\beta_{\mathrm{min}}^\pm(s))$ according to Eq.~\eqref{ConstructionConvexEnvelope} and we use the same construction for $\alpha_{\mathrm{max}}^\pm$ and $\beta_{\mathrm{max}}^\pm$. Note, that $f_\mathrm{a}$ is affine along these connecting lines and the area, which is covered by those lines consists of four triangular shaped sets denoted by $Y_2$ in Figure~\ref{fig:constructionrelaxation}. The fact that $\alpha_{\mathrm{min}}^\pm$ and $\alpha_{\mathrm{max}}^\pm$ intersect in $y_\ast^\pm$ means that the tangential plane of $f$ in $y_\ast^\pm$ touches the graph in two other points, namely
\begin{align}
y_{\ast \mathrm{min}}^\pm&:= \beta_{\mathrm{min}}^\pm(s_\ast)= (y_{\mathrm{min}},\pm \frac{b+1}{\sqrt{b}}(y_{\mathrm{max}}- y_{\mathrm{min}})/2)\,,\\
y_{\ast \mathrm{max}}^\pm&:= \beta_{\mathrm{max}}^\pm(s_\ast)= (y_{\mathrm{max}},\pm \frac{b+1}{\sqrt{b}}(y_{\mathrm{max}}- y_{\mathrm{min}})/2)\,.
\end{align}
Since the pairs $(\alpha_{\mathrm{min}}^\pm(s),\beta_{\mathrm{min}}^\pm(s))$ and $(\alpha_{\mathrm{max}}^\pm(s),\beta_{\mathrm{max}}^\pm(s))$ satisfy the equality in Eq.~\eqref{SubgradientCondition} for all $s\in (0,s_\ast)$, the triples $(y_\ast^\pm,y_{\ast \mathrm{min}}^\pm,y_{\ast \mathrm{max}}^\pm)$ also satisfy Eq.~\eqref{SubgradientCondition}. Therefore we define $f_\mathrm{a}$ for $y$ in one of the triangles $\text{Conv}(\{y_\ast^+, y_{\ast \mathrm{min}}^+, y_{\ast \mathrm{max}}^+\})\cup \text{Conv}(\{y_\ast^-, y_{\ast \mathrm{min}}^-, y_{\ast \mathrm{max}}^-\})$ ($Y_3$ in Figure~\ref{fig:constructionrelaxation}) according to Eq.~\eqref{ConstructionConvexEnvelopeDifferentiable} as the unique affine function which coincides with $f$ in the vertices of each of these triangles,
\begin{align} \label{ConstructionOnY3}
 f_\mathrm{a}(y)= \left\{ \begin{array}{cl} (T_{y_\ast^+}f)(y)&\text{ on }\text{Conv}(\{y_\ast^+, y_{\ast \mathrm{min}}^+, y_{\ast \mathrm{max}}^+\})\,,\\ (T_{y_\ast^-}f)(y)&\text{ on }\text{Conv}(\{y_\ast^-, y_{\ast \mathrm{min}}^-, y_{\ast \mathrm{max}}^-\})\,. \end{array} \right.
\end{align}
Finally, consider a pair of two different points $(\widetilde{y}^{(1)}, \widetilde{y}^{(2)})\in \partial \widetilde{Y}$ satisfying Eq.~\eqref{SubgradientCondition}. For the existence of a supporting hyperplane of $f$, which touches the graph of $f$ in $\widetilde{y}^{(1)}$ and $\widetilde{y}^{(2)}$, we require equal partial derivatives in the second component, i.e.,
\begin{align}
\widetilde{y}^{(1)}_2= \frac{b+1}{b} \partial_2 f(\widetilde{y}^{(1)})= \frac{b+1}{b} \partial_2 f(\widetilde{y}^{(2)})= \widetilde{y}^{(2)}_2\,.
\end{align}
With $\widetilde{y}^{(1)}_1, \widetilde{y}^{(2)}_1\in \{y_{\mathrm{min}},y_{\mathrm{max}}\}$ and $\widetilde{y}^{(1)}\neq \widetilde{y}^{(2)}$ we get $\widetilde{y}^{(1)}= y_{\mathrm{min}}$ and $\widetilde{y}^{(2)}_1= y_{\mathrm{max}}$ and obtain for $s> 0$ the pairs of curves
\begin{align}
\alpha_\infty^\pm(s):= (y_{\mathrm{min}},\pm s)\,,\quad
\beta_\infty^\pm(s):= (y_{\mathrm{max}},\pm s)\,.
\end{align}
Since we already constructed our candidate $f_\mathrm{a}$ for the convex envelope on the set $(y_{\mathrm{min}},y_{\mathrm{max}})\times [- \frac{b+1}{\sqrt{b}} s_\ast, \frac{b+1}{\sqrt{b}} s_\ast]$, we restrict $s$ to have values $s> \frac{b+1}{\sqrt{b}} s_\ast$. For each $y\in \R^2$, which lies on a connecting straight line of two corresponding points of the curves $\alpha_\infty^\pm$ and $\beta_\infty^\pm$, i.e., $y= t\cdot \alpha_\infty^\pm(s)+ (1-t)\cdot \beta_\infty^\pm(s)$ with $s\in (\frac{b+1}{\sqrt{b}} s_\ast,\infty)$ and $t\in [0,1]$ (represented by the two unbounded sets forming $Y_4$ in Figure~\ref{fig:constructionrelaxation}), we define $f_\mathrm{a}(y)= t\cdot f(\alpha_\infty^\pm(s))+ (1-t)\cdot f(\beta_\infty^\pm(s))$ according to Eq.~\eqref{ConstructionConvexEnvelope}. We summarize the result of the construction in the following theorem; the explicit formulas in the different regions will be verified in the proof.

\begin{theorem} \label{thm:RelaxationSmallb}
Define 
\begin{align}
 Y_1 &= \{y\in \R^2,\ |y_2|< r_0(y_1)\},\\
 Y_2& = \{y\in (y_{\mathrm{min}},y_{\mathrm{max}})\times \R,\ r_0(y_1)\leq |y_2|< \frac{b+1}{\sqrt{b}} \frac{y_{\mathrm{max}}-y_{\mathrm{min}}}{2}- \frac{r_0(y_1)}{b}\},\\
 Y_3& = \{y\in (y_{\mathrm{min}},y_{\mathrm{max}})\times \R,\ \frac{b+1}{\sqrt{b}} \frac{y_{\mathrm{max}}-y_{\mathrm{min}}}{2}- \frac{r_0(y_1)}{b}\leq |y_2|\leq \frac{b+1}{\sqrt{b}} \frac{y_{\mathrm{max}}-y_{\mathrm{min}}}{2}\},\\
 Y_4& = \{y\in (y_{\mathrm{min}},y_{\mathrm{max}})\times \R,\ \frac{b+1}{\sqrt{b}} \frac{y_{\mathrm{max}}-y_{\mathrm{min}}}{2}< |y_2|\},
\end{align}
as shown in Figure~\ref{fig:constructionrelaxation}. Then the convex envelope $f_\mathrm{c}$ of $f$ is given by 
\begin{align}
 f_\mathrm{c}(y) = \left\{\begin{array}{cl} 
\displaystyle \frac{1}{2}y_1^2+ \frac{1}{2}\frac{b}{b+1}y_2^2 & \text{ on }\widetilde{Y}\,,\\[.1in]
\displaystyle \frac{1}{2}y_1^2+ \frac{1}{2}y_2^2 & \text{ on } Y_1\,, \\[.1in]
\displaystyle \frac{1}{2}y_1^2 + \frac{1}{2}y_2^2- \frac{(|y_2|- r_0(y_1))^2}{2(b+1)} & \text{ on }Y_2\,,\\[.1in]
\displaystyle  \begin{aligned} &\frac{1}{2}y_1^2+ \frac{1}{2}y_2^2- \frac{(|y_2|- r_0(y_1))^2}{2(b+1)}\\ 
\displaystyle &\qquad - \frac{1}{2} \frac{b}{b+1} \left( |y_2|- \frac{b+1}{\sqrt{b}} \frac{y_{\mathrm{max}}-y_{\mathrm{min}}}{2}+ \frac{r_0(y_1)}{b} \right)^2 \end{aligned} & \text{ on }Y_3\,,\\[.1in]
\displaystyle \frac{1}{2}y_1^2+ \frac{1}{2}\frac{b}{b+1}y_2^2+ \frac{1}{2}(y_1- y_{\mathrm{min}})\cdot (y_{\mathrm{max}}- y_1)& \text{ on }Y_4\,.
\end{array} \right.
\end{align}

\end{theorem}

\begin{proof}
Denote by $f_\mathrm{a}$ the formula on the right-hand side in the assertion of the theorem. We are going to show in Step~1 that $f_\mathrm{a}$ is differentiable on $\R^2\setminus (\{y_{\mathrm{min}},y_{\mathrm{max}}\}\times \R)$ and convex. In Step~2 we verify that $f_\mathrm{a}$ is indeed the expression we obtain from the previously described construction and $f_\mathrm{c}\leq f_\mathrm{a}$. Finally in Step~3 we prove $f_\mathrm{a}\leq f$ and conclude, that $f_\mathrm{a}$ coincides with $f_\mathrm{c}$ by $f_\mathrm{c}\leq f_\mathrm{a}\leq f$ and $f_\mathrm{a}$ being convex.
\medskip

\textit{Step 1: Differentiability and convexity of $f_\mathrm{a}$.} The differentiability of $f_\mathrm{a}$  on $\interior(\widetilde{Y})\cup (y_{\mathrm{min}},y_{\mathrm{max}})\times \R$ follows from 
\begin{align}
\nabla f_\mathrm{a}(y) = \left\{\begin{array}{cl} 
\displaystyle (y_1, \frac{b}{b+1}y_2) & \text{ on }\mathring{\widetilde{Y}}\,,\\[.1in]
\displaystyle (y_1, y_2) & \text{ on } Y_1\,, \\[.1in]
\displaystyle (y_1+ \frac{\sign(y_{\mathrm{mid}}- y_1) \sqrt{b}}{b+1}(|y_2|- r_0(y_1)), \frac{b}{b+1} y_2+ \frac{\sign(y_2) r_0(y_1)}{b+1}) & \text{ on }Y_2\,,\\[.1in]
\displaystyle (y_{\mathrm{mid}}, \sign(y_2) \sqrt{b}\cdot \frac{y_{\mathrm{max}}-y_{\mathrm{min}}}{2}) & \text{ on }Y_3\,,\\[.1in]
\displaystyle (y_{\mathrm{mid}}, \frac{b}{b+1}y_2) & \text{ on }Y_4\,,\end{array} \right.
\end{align}
and verifying the continuity of $\nabla f_\mathrm{a}$ along the boundaries of the different regions. With the notation $e_1\odot e_2:= \frac{1}{2} (e_1\otimes e_2+ e_2\otimes e_1)$, the second derivatives are given by
\begin{align}
\nabla^2 f_\mathrm{a}(y) = \left\{\begin{array}{cl} 
\displaystyle e_1\otimes e_1+ \frac{b}{b+1}e_2\otimes e_2 & \text{ on }\interior(\widetilde{Y})\,,\\[.1in]
\displaystyle e_1\otimes e_1+ e_2\otimes e_2 & \text{ on } \interior(Y_1)\,, \\[.1in]
\displaystyle \frac{1}{b+1}\cdot \left( e_1\otimes e_1+ \frac{2 \sqrt{b} \sign(y)}{\sign(y_{\mathrm{mid}}- y_1)}e_1\odot e_2 + b e_2\otimes e_2 \right) & \text{ on }\interior(Y_2)\,,\\[.1in]
\displaystyle 0 & \text{ on }\interior(Y_3)\,,\\[.1in]
\displaystyle \frac{b}{b+1}e_2\otimes e_2 & \text{ on }\interior(Y_4)\,.
\end{array} \right.
\end{align}
In fact, $\nabla^2 f_\mathrm{a}$ is positively semidefinite on the interior of each domain $\widetilde{Y}, Y_1, Y_2, Y_3, Y_4$, consequently $f_\mathrm{a}$ is locally convex on those domains. Since $f_\mathrm{a}$ is differentiable on $(y_{\mathrm{min}},y_{\mathrm{max}})\times \R$, an iterative application of Lemma~\ref{lem:ConvexUnionSets} gives us the local convexity (and hence the convexity by Lemma~\ref{lem:LocallyConvexIsGloballyConvex}) of $f_\mathrm{a}$ on $(y_{\mathrm{min}},y_{\mathrm{max}})\times \R$. A short calculation shows, that $g: \R^2\to \R,\ g(y)= \frac{1}{2}y_1^2+ \frac{1}{2}\frac{b}{b+1}y_2^2$ is a convex lower bound of $f_\mathrm{a}$, and since $f_\mathrm{a}$ and $g$ coincide on $\widetilde{Y}$ and $f_\mathrm{a}$ is convex on $\R\setminus \widetilde{Y}= (y_{\mathrm{min}},y_{\mathrm{max}})\times \R$, the convexity of $f_\mathrm{a}$ follows by Lemma~\ref{lem:ConvexityOnConvexSubsetSufficient}.

\textit{Step 2: $f_\mathrm{c}\leq f_\mathrm{a}$.} Since the calculation of the explicit formula of the candidate $f_\mathrm{a}$ derived from the previously described construction is extensive, we just prove, that the given formula is the one, which is obtained by this construction. Refer to Figure~\ref{fig:constructionrelaxation} for a sketch of the various domains.

\textit{On $\widetilde{Y}\cup Y_1$:}
Since by assumption $r_0\leq r$ and hence $Y_1\subset Y_0$, $f_\mathrm{a}$ coincides with $f$ in $Y_1$ and in $\widetilde{Y}$ and $f_\mathrm{c}\leq f= f_\mathrm{a}$, as asserted.

\textit{On $Y_2$:}
The set $Y_2$ is exactly the set of points, which are covered by the straight line segments connecting the corresponding pairs of points on the four pairs of curves $(\alpha_{\mathrm{min}}^\pm, \beta_{\mathrm{min}}^\pm)$ and $(\alpha_{\mathrm{max}}^\pm,\beta_{\mathrm{max}}^\pm)$. Assume $y\in Y_2$ with $y_1< y_{\mathrm{mid}}$ and $y_2
> 0$ (the other cases are analogous). There exists $s\in (0,(y_{\mathrm{max}}- y_{\mathrm{min}})/2)$ and $t\in (0,1)$ with $y= t\cdot \alpha_{\mathrm{min}}^+(s)+ (1-t)\cdot \beta_{\mathrm{min}}^+(s)$. To show that the given expression for $f_\mathrm{a}(y)$ results from the construction described above, we notice that $f_\mathrm{a}(\alpha_{\mathrm{min}}^+(s))= f(\alpha_{\mathrm{min}}^+(s))$ and $f_\mathrm{a}(\beta_{\mathrm{min}}^+(s))= f(\beta_{\mathrm{min}}^+(s))$ and that $f_\mathrm{a}$ is affine along the connecting line between $\alpha_{\mathrm{min}}^+(s)$ and $\beta_{\mathrm{min}}^+(s)$. This can be seen by $\beta_{\mathrm{min}}^+(s)- \alpha_{\mathrm{min}}^+(s)=(s-y_{\mathrm{min}})\cdot (-1,\sqrt{b}^{-1})^T$ and using $y_2$, $y_{\mathrm{mid}}- y_1> 0$ to obtain
\begin{align}
(-1,\sqrt{b}^{-1}) \nabla^2 f_\mathrm{a}(y) \begin{pmatrix} -1 \\ \sqrt{b}^{-1} \end{pmatrix}= \frac{1}{b+1} (-1,\sqrt{b}^{-1}) \begin{pmatrix} 1 & \sqrt{b} \\ \sqrt{b} & b \end{pmatrix} \begin{pmatrix} -1 \\ \sqrt{b}^{-1}\end{pmatrix}= 0\,.
\end{align}
Since $f_\mathrm{c}$ must lie below the linear interpolation, we obtain $f_\mathrm{c}\leq f_\mathrm{a}$.

\textit{On $Y_3$:}
By construction, $f_\mathrm{a}$ and $f$ have the same values in $y_\ast^\pm, y_{\ast \mathrm{min}}^\pm, y_{\ast \mathrm{max}}^\pm$ and on $\text{Conv}(y_\ast^+, y_{\ast \mathrm{min}}^+, y_{\ast \mathrm{max}}^+)$ and $\text{Conv}(y_\ast^-, y_{\ast \mathrm{min}}^-, y_{\ast \mathrm{max}}^-)$ the function $f_\mathrm{a}$ is affine due to $\nabla^2 f_\mathrm{a}= 0$ on $\interior(Y_3)$. By $y_\ast^\pm\in Y_0$, $f$ is differentiable in $y_\ast^\pm$ with $\nabla f(y_\ast^\pm)= \nabla f_\mathrm{a}(y_\ast^\pm)$. Now it follows in view of Eq.~\eqref{ConstructionOnY3} that $f_\mathrm{a}$ is the above constructed function.

\textit{On $Y_4$:}
Similar to the argument on $Y_2$, we recognize that for any $y\in Y_4$ with $\pm y_2> 0$ there is an $s\in (0,\frac{b+1}{\sqrt{b}}(y_{\mathrm{max}}- y_{\mathrm{min}})/2)$ and $t\in (0,1)$ with $y= t\cdot \alpha^\pm(s)+ (1-t)\cdot \beta^\pm(s)$. We have $f_\mathrm{a}(\alpha_\infty^\pm(s))= f(\alpha_\infty^\pm(s))$ and $f_\mathrm{a}(\beta_\infty^\pm(s))= f(\beta_\infty^\pm(s))$, while $f_\mathrm{a}$ is affine along the connecting line between $\alpha_\infty^\pm(s)$ and $\beta_\infty^\pm(s)$ since $\beta_\infty^\pm(s)- \alpha_\infty^\pm(s)= (y_{\mathrm{max}}-y_{\mathrm{min}},0)$ and in the formula of $f_\mathrm{a}$ on $Y_4$ the quadratic term in $y_1$ cancels out.
Finally, we get $f_\mathrm{c}\leq f_\mathrm{a}$ since any value of $f_\mathrm{a}$ is a convex combination of two or three function values of $f$.

\medskip

\textit{Step 3: The inequality $f_\mathrm{a}\leq f$.} We show $f_\mathrm{a}\leq f^{(r_0)}\leq f$, where the second inequality is an immediate consequence of Lemma~\ref{lem:OrderCondensedEnergy}. For the first inequality, recognize, that $f_\mathrm{a}= f^{(r_0)}$ holds on $\widetilde{Y}\cup Y_1\cup Y_2$ and the inequality only has to be shown on $Y_3$ and $Y_4$.\\
On $Y_3$ the estimate follows by $f_\mathrm{a}(y)= f^{(r_0)}(y)- \frac{1}{2} \frac{b}{b+1} \left( |y_2|- \frac{b+1}{\sqrt{b}} \frac{y_{\mathrm{max}}-y_{\mathrm{min}}}{2}+ \frac{r_0(y_1)}{b} \right)^2$.
On $Y_4$ we can calculate
\begin{align}
f^{(r_0)}(y)- \left( \frac{1}{2}y_1^2+ \frac{1}{2}\frac{b}{b+1}y_2^2 \right)
&= \frac{y_2^2}{2(b+1)}- \frac{(|y_2|- r_0(y_1))_+^2}{2(b+1)}
\geq - \frac{r_0(y_1)^2}{2(b+1)}+ \frac{r_0(y_1)}{b+1}|y_2|\\
&\geq - \frac{r_0(y_1)^2}{2b}+ \frac{r_0(y_1)}{\sqrt{b}}\frac{y_{\mathrm{max}}-y_{\mathrm{min}}}{2}\,.
\end{align}
For $y_1\in (y_{\mathrm{min}},y_{\mathrm{mid}}]$ we have
\begin{align}
- \frac{r_0(y_1)^2}{2b}+ \frac{r_0(y_1)}{\sqrt{b}}\frac{y_{\mathrm{max}}-y_{\mathrm{min}}}{2}
&= - \frac{b(y_1-y_{\mathrm{min}})^2}{2b}+ \frac{\sqrt{b}(y_1-y_{\mathrm{min}})}{\sqrt{b}}\cdot \frac{y_{\mathrm{max}}-y_{\mathrm{min}}}{2}\\
&= \frac{1}{2}(y_1-y_{\mathrm{min}})(y_{\mathrm{max}}-y_1)
\end{align}
and for $y_1\in (y_{\mathrm{mid}},y_{\mathrm{max}})$ we have
\begin{align}
- \frac{r_0(y_1)^2}{2b}+ \frac{r_0(y_1)}{\sqrt{b}}\frac{y_{\mathrm{max}}-y_{\mathrm{min}}}{2}
&= - \frac{b(y_{\mathrm{max}}-y_1)^2}{2b}+ \frac{\sqrt{b}(y_{\mathrm{max}}-y_1)}{\sqrt{b}}\cdot \frac{y_{\mathrm{max}}-y_{\mathrm{min}}}{2}\\
&= \frac{1}{2}(y_1-y_{\mathrm{min}})(y_{\mathrm{max}}-y_1)\,,
\end{align}
which gives us the desired estimate on $Y_4$.
\end{proof}

\begin{figure}
\centering 
\includegraphics[width=.45\textwidth]{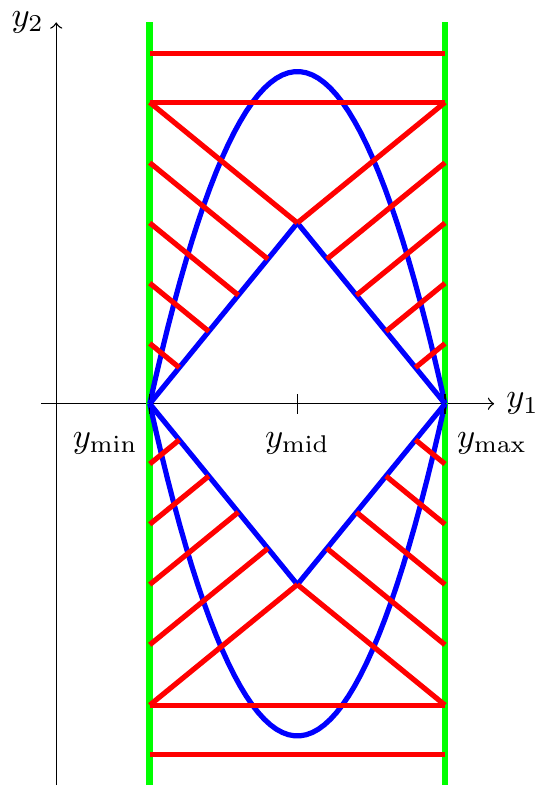}
\hfil
\includegraphics[width=.45\textwidth]{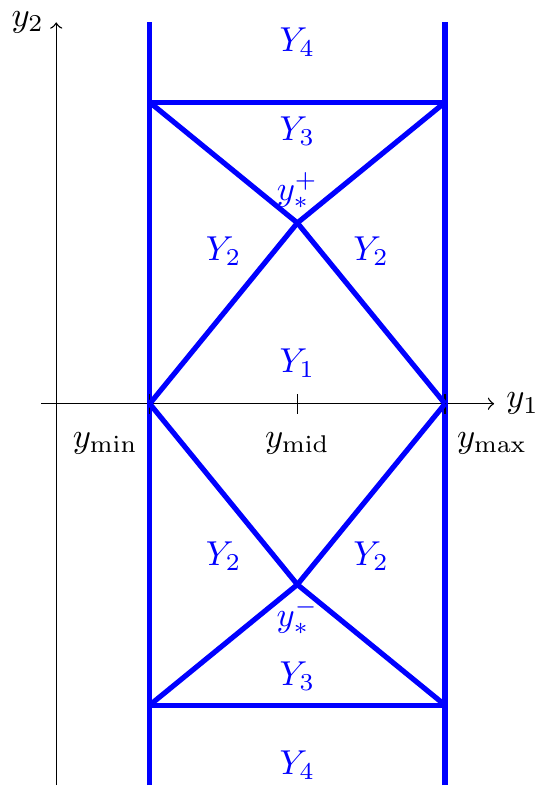}
\caption{Construction of the relaxed energy}\label{fig:constructionrelaxation}
\end{figure}

\begin{corollary} \label{cor:RelaxationSmallb}
For all $r\in \mathcal{R}$ with $\sqrt{b} \cdot \frac{y_{\mathrm{max}}-y_{\mathrm{min}}}{2} \leq r\Barg{ \frac{y_{\mathrm{min}}+y_{\mathrm{max}}}{2} }$ the equality
\begin{align}
f_\mathrm{c}^{(r)}= f_\mathrm{c}^{(r_0)}
\end{align}
holds.
\end{corollary}
\begin{proof}
In the proof of the previous theorem we already showed the inequality 
\begin{align}
f_\mathrm{c}^{(r)}\leq f^{(r_0)}\leq f^{(r)}\,,
\end{align}
which implies the claimed equality by the monotonicity in taking convex envelopes.
\end{proof}

\begin{remark}
More generally, one can even show that without any restrictions on $r$ the convex envelope of $f^{(r)}$ does not change while replacing $r$ by $\min\{r,r_0\}$. Corollary~\ref{cor:RelaxationSmallb} is a special case of this statement, since Eq.~\eqref{ConditionSmallb} implies $\min\{r,r_0\}= r_0$ by concavity of $r$ on $(y_{\mathrm{min}}, y_{\mathrm{max}})$.
\end{remark}

\section{Extension of the results to three dimensions}
\label{sec6}

The convex envelope in the one-dimensional case for the density $f(y_1,y_2,\varepsilon_{\mathrm{p},n},p_n)$ given in Theorem \ref{thm:RelaxationSmallb} may be generalized in a natural way to three dimensions by employing a substitution inspired by the transformation of variables defined in Eq.~\eqref{eq26}. In this sense, we define a formally relaxed energy as
\be
\label{eqa1}
\psi_\mathrm{rel}(\bfeps,\bfeps_{\mathrm{p},n},p_n) = f\left(\sqrt{\frac{K}{2\mu}} 
\tr\bfeps,\|\dev\bfeps\|,\bfeps_{\mathrm{p},n},p_n\right).
\ee
Observe that, when applying the same substitution to the one-dimensional condensed energy given in Eq.~\eqref{eq27}, we recover the original three-dimensional condensed energy given in Eq.~\eqref{eq17}.

The relations of the relaxed energies according to Fig.~\ref{fig:constructionrelaxation}, and their derivatives with respect to the strain tensor, i.e., the corresponding relaxed stress tensor, are given below.

Domain $\mathrm{Y_1} $:

\be
\label{eqa3}
\psi_\mathrm{1,rel}= \frac{K}{2} (\tr \bfeps)^2 \,+\, \mu\, \|\, \dev\bfeps \|^2, \quad \sigma_\mathrm{1,rel} = K (\tr \bfeps)\,\mathrm{I} \,+\, 2\mu\,\dev\bfeps \,. \hfill
\ee
Domains $\mathrm{Y_2}$:

\begin{align}\label{eqa5}
	\psi_\mathrm{2,rel}& = \frac{K}{2}(\tr \bfeps)^2 \,+\, \mu\, \| \dev\bfeps\|^2  - \frac{4\mu^2}{2(2\mu+\beta)} \Big(\|\dev\bfeps\| - \frac{1}{2\mu}\rho_0(\tr \bfeps)\Big)^2\,,\\
&\begin{aligned}
\label{eqa5.5}
\sigma_\mathrm{2,rel}  =K (\tr \bfeps)\,\mathrm{I} \,+\, 2\mu\,\dev\bfeps - \,\frac{4\mu^2}{2\mu+\beta}\Big(\|\dev\bfeps\| - \frac{1}{2\mu}\rho_0(\tr\bfeps)\Big)\frac{\dev \bfeps}{\|\dev \bfeps\|} \\ \qquad 
-\,\frac{2\mu\,\sqrt {\beta\, K}}{2\mu+\beta}\Big(\|\dev\bfeps\| - \frac{1}{2\mu}\rho_0(\tr\bfeps)\Big)\,\mathrm{I}\,.
\end{aligned}
\end{align}

Domains $\mathrm{Y_3}$:

\begin{multline}
\label{eqa6}
	\psi_\mathrm{3,rel} =   \frac{K}{2}(\tr \bfeps)^2 \,+\, \mu\, \| \dev\bfeps\|^2  - \frac{4\mu^2}{2(2\mu+\beta)} \Big(\|\dev\bfeps\| - \frac{1}{2\mu}\rho_0(\tr\bfeps)\Big)^2 \\
 - \frac{2\mu \beta}{2(2\mu+\beta)}\Big(\|\dev\bfeps\| + \frac{1}{\beta}\,\rho_0(\tr\bfeps)
  -\frac{2\mu+\beta}{2 \mu} \,\sqrt{\frac{K}{\beta}}\,\frac{\tr \bfeps_\mathrm{max} -  \tr\bfeps_\mathrm{min}}{2} \Big)^2\,,
\end{multline}
\begin{multline}
\label{eqa6.5}
	\sigma_\mathrm{3,rel} = K (\tr \bfeps)\,\mathrm{I} \,+\, 2\mu\,\dev\bfeps - \,\frac{4\mu^2}{2\mu+\beta}\Big(\|\dev\bfeps\| - \frac{1}{2\mu}\rho_0(\tr\bfeps)\Big)\frac{\dev \bfeps}{\|\dev \bfeps\|} \\
-\,\frac{2\mu\,\sqrt {\beta\, K}}{2\mu+\beta}\Big(\|\dev\bfeps\| - \frac{1}{2\mu}\rho_0(\tr\bfeps)\Big)\,\mathrm{I}\\
  -\frac{2\mu\sqrt {\beta\, K}}{2\mu+\beta}\Big(\|\dev\bfeps\| + \frac{1}{\beta}\,\rho_0(\tr\bfeps) 
  -\frac{2\mu+\beta}{2 \mu} \,\sqrt{\frac{K}{\beta}}\,\frac{\tr \bfeps_\mathrm{max} -  \tr\bfeps_\mathrm{min}}{2} \Big)\,\mathrm{I} \\
 - \,\frac{2\mu\beta}{2\mu+\beta}\Big(\|\dev\bfeps\| + \frac{1}{\beta}\,\rho_0(\tr\bfeps) 
  -\frac{2\mu+\beta}{2 \mu} \,\sqrt{\frac{K}{\beta}}\,\frac{\tr \bfeps_\mathrm{max} -  \tr\bfeps_\mathrm{min}}{2}\Big)\frac{ \dev \bfeps}{\| \dev\bfeps   \|}\,.
\end{multline}

Domains $\mathrm{Y_4}$:

\begin{align}
\label{eqa7}
	\psi_\mathrm{4,rel} =  \frac{K}{2 } (\tr \bfeps)^2  + \frac{2 \mu \beta}{2(2\mu+\beta)} \| \dev\bfeps\|^2 + \frac{K}{2} \Big(\tr \bfeps -  \tr\bfeps_\mathrm{min}\Big)\Big(\tr\bfeps_\mathrm{max} - \tr \bfeps \Big)\,, \hfill
\end{align}

\begin{multline}
\label{eqa7.5}
	\sigma_\mathrm{4,rel} = K (\tr \bfeps)\,\mathrm{I} + \frac{2\mu\beta}{2\mu+\beta} \dev \bfeps + K\Big(\frac{\tr\bfeps_\mathrm{max} + \tr\bfeps_\mathrm{min}}{2} - \tr \bfeps \Big)\,\mathrm{I}\,. \hfill
\end{multline}

for

\begin{multline}
\rho_0(\tr\bfeps) = \sqrt{\beta K} \Big(\frac{\tr \bfeps_\mathrm{max} -  \tr\bfeps_\mathrm{min}}{2} - |\tr(\bfeps) - \frac{\tr \bfeps_\mathrm{max} + \tr\bfeps_\mathrm{min}}{2} |\Big) \hfill
\end{multline}

\section{Numerical results}
\label{sec7}

\subsection{The one-dimensional model problem}

We make use of the dimensionless formulation introduced in Eq.~\eqref{eq26}. The function $r(y_1)$ is selected to fit the yield surface of a modified Drucker-Prager model employing two quadratic functions over two intervals as given below.
\begin{align}
	r(y_1)= r_\mathrm{max}
	\left\{ \begin{array}{lll} 
		\D 1-\frac{\big(y_1 - y_0\big)^2}{\big(y_0 - y_\mathrm{min}\big)^2}\; &\text{for} & y_\mathrm{min}\leq y_1 \leq y_0 \\
		\\
		\D 1 - \frac{\big(y_1 - y_0\big)^2}{\big(y_\mathrm{max} - y_0 \big)^2}\; &\text{for} & y_0\leq y_1\leq y_\mathrm{max} \\ \\
		0 & \text{otherwise}
	\end{array} \right. ,
\end{align}
where $ y_\mathrm{min} = -0.058 $, $y_\mathrm{max} = 0.00107 $, $ y_0 = -0.0385$ and $r_\mathrm{max} = 0.016 $, see Fig.~\ref{fig-r(xi)}. We choose a hardening parameter $b=0.095$ to insure that we are within the regime of small values of $b$ treated in Sec.~\ref{sec5}.

\begin{figure}
	\centering
	\includegraphics[scale=1.2]{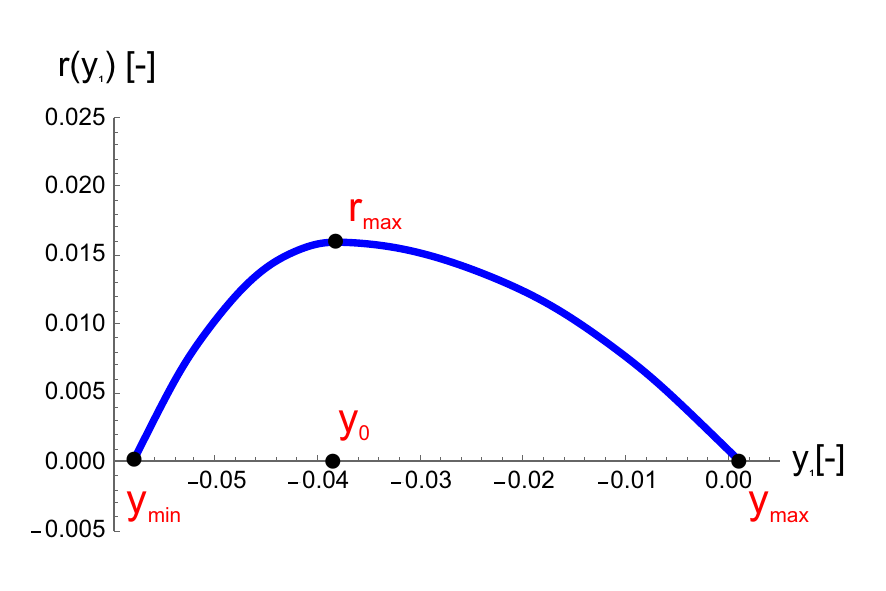} 
	\caption{Elastic region given by the function $r(y_1)$.}
	\label{fig-r(xi)}
\end{figure}

We model a one-dimensional specimen of length $L=1$, represented by a domain $\Omega = (0,1)$, employing a discretization using $n=80$ standard linear finite elements.
The specimen is fixed at the left-hand side, $u(0)=v(0)=0$, and subjected to given displacements $u(L)=u_\mathrm{ext}$ and $v(L)=v_\mathrm{ext}$ in $x$- and $y$-direction, respectively, at the right-hand side. The affine solution is given by $x\mapsto (u(x), v(x))=x(u_\mathrm{ext}, v_\mathrm{ext})$ with constant derivative $(u_{\mathrm{ext}},v_{\mathrm{ext}})$. We fix $u_{\mathrm{ext}}$ throughout our calculations and vary $v_{\mathrm{ext}}$ exploring four different regions in the phase diagramme Fig.~\ref{fig:constructionrelaxation} which correspond to the four panels in Fig.~\ref{fig-out1}: $(u_{\mathrm{ext}}, v_{\mathrm{ext}})$ in (a) $Y_1$, (b) $Y_2\cap\{y_2\leq r(y_1)\}$, (c) $Y_2\setminus\{y_2\leq r(y_1)\}$ and (d) $Y_3$, respectively.

The total energies are then functions of the vectors of nodal displacements $\bfm{u}=\{u_1,\ldots,u_{n+1}\}$, $\bfm{v}=\{v_1,\ldots,v_{n+1}\}$ with $u_1 = v_1 = 0$ and $u_{n+1} = u_\mathrm{ext}$, $v_{n+1} = v_\mathrm{ext}$, given as
\begin{equation} 
	\label{eq50}
	W_\mathrm{cond,rel}(\bfm{u},\bfm{v}) 
	=
	\sum_{j=1}^{n} \psi_\mathrm{cond,rel}\Big(\frac{u_{j+1} - u_j}{L/n},\frac{v_{j+1} - v_j}{L/n}\Big) \,\frac{L}{n} ,
\end{equation}
for the condensed and the relaxed energy, respectively. The minimization is performed using the \textit{Mathematica} function \textit{FindMinimum}, which employs a gradient-based local search. Note, that for the relaxed energy, the minimum is obtained for the affine displacements $u_i=\frac{i-1}{n} u_\mathrm{ext}$, $v_i=\frac{i-1}{n} v_\mathrm{ext}$. 

We add a random perturbation of amplitude $\alpha$ to the exact solution given above as an initial guess for the minimization procedure. Hence, let $\mathrm{rand}(i)$ be a random function possessing a uniform distribution within the interval $[-1,1]$. Then the initial guess is
\begin{multline}
\label{eq50a}
u^\mathrm{ini}_i = \frac{i-1}{n} u_\mathrm{ext} +\alpha \frac{L}{n} \mathrm{rand}(i), \quad v^\mathrm{ini}_i = \frac{i-1}{n} v_\mathrm{ext}+\alpha \frac{L}{n}\mathrm{rand}(n+i),  i=2,\ldots,n.
\end{multline}
Results are depicted in Fig.~\ref{fig-out1}. The numerical simulations reflect the predictions based on the phase diagram in Figure~\ref{fig:constructionrelaxation}.
In case (a) with $v_\mathrm{ext}\in [0,0.0036]$, both the relaxed and the condensed energy are convex with $f=f_{\mathrm{c}}$ and the corresponding minimizers are affine with gradient $(u_{\mathrm{ext}},v_{\mathrm{ext}})$. For larger values of $v_{\mathrm{ext}}\in [0.0036, 0.014]$ the function $f$ is locally convex but different from its relaxation $f_{\mathrm{c}}$ which is obtained by a mixture of two states and which is affine on lines parallel to the boundary of the region $Y^4$, see the red lines in the left panel in Figure~\ref{fig:constructionrelaxation}. Thus the calculation with the condensed energy is expected to discover the microstructure supported on two points while the calculation with $f_{\mathrm{c}}$ can use all the gradients on the affine part. In fact, the second panel in Figure~\ref{fig-out1} shows exactly this behavior. The case (c) corresponds to $v_{\mathrm{ext}}\in [ 0.014, 0.063]$ and shows the same  behavior of case (b), the only difference being that $f$ is not locally convex at the point $(u_{\mathrm{ext}},v_{\mathrm{ext}})$. Finally, for $v_{\mathrm{ext}}\in [0.0163, 0.15]$ the condensed energy $f$ is not locally convex and the relaxed energy $f_{\mathrm{c}}$ is affine on the triangle $Y_3$. Therefore the range of the gradient of the numerical solution for the calculation with the condensed energy is expected to be located in the three corners of the triangle while the gradient of the solution with the relaxed energy can explore the full triangle. The last panel in Fig.~\ref{fig-out1} confirms this prediction. 

The procedure employing the condensed energy does not recover the exact minimum, but produces values which are between 10\% and 20\% higher. However, the values are lower than the energy of the condensed energy evaluated at the affine solution. The approximate minimizer attempts to mimic the construction of the relaxed energy by concentrating points at the corresponding positions within the  $y_1$-$y_2$-plane.
The procedure employing the relaxed energy returns the exact minimum with working precision in all cases.

\begin{figure}
	\centering
	\includegraphics[scale=0.75]{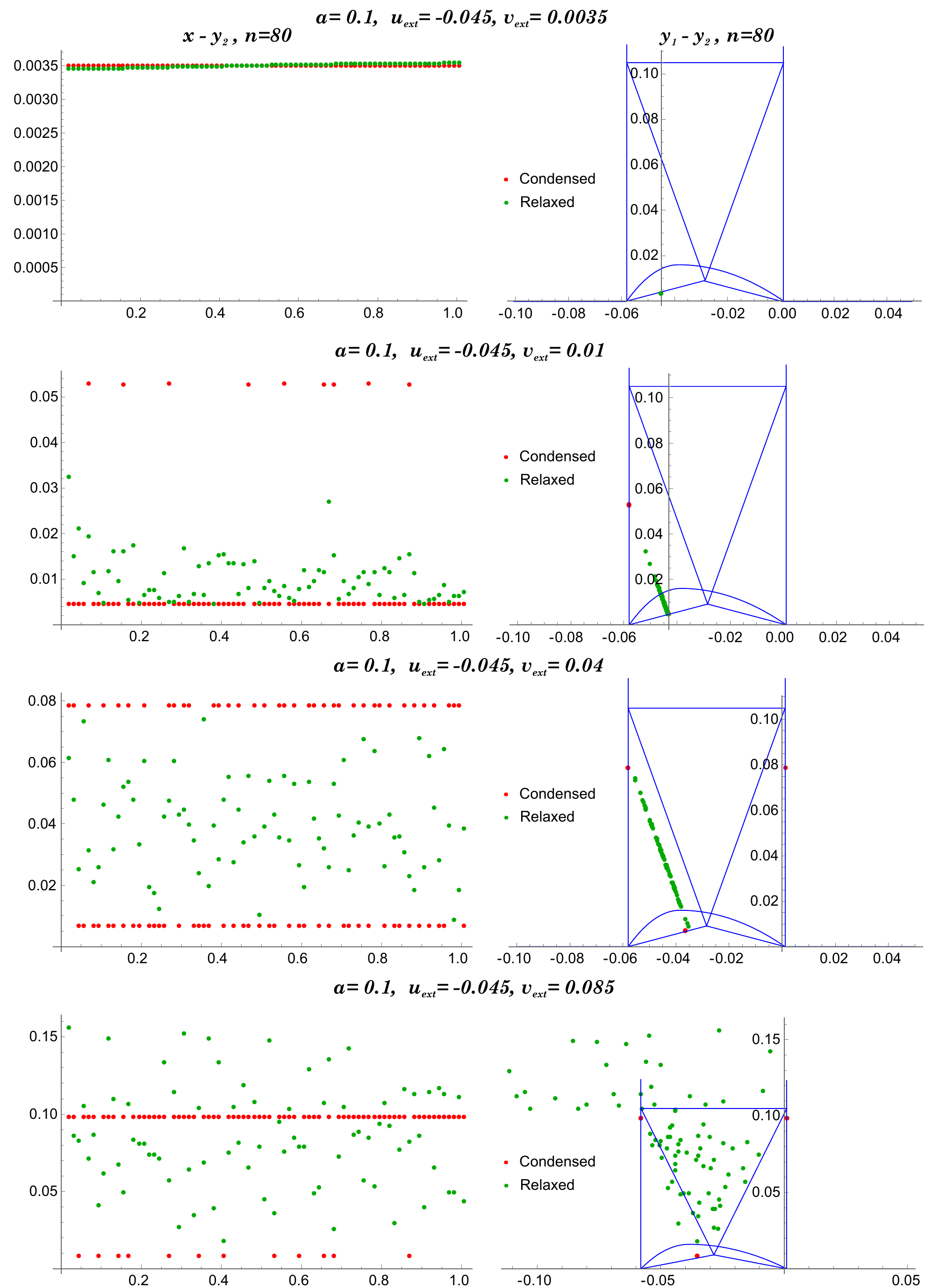}  
	\caption{Minimization with exact initial guess plus random perturbation for increasing boundary values $u_\mathrm{ext}$, $v_\mathrm{ext}$. Left column: $y_2=v^\prime(x)$ versus position $x$ of the beam, right column: points of the approximate minimizers $(y_1, y_2)$ within relaxation domains as depicted in Fig.~\ref{fig:constructionrelaxation}. Comparison from the condensed and relaxed energies for a mesh with 80 elements.}
	\label{fig-out1}
\end{figure}

\subsection{Three-dimensional model}

The generalization to a three-dimensional model derived in Sec.~\ref{sec6} is tested through a two-dimensional boundary value problem, namely the square plate with a circular hole.  The plate is fixed at the left-hand side and subject to displacements in compression in $x$-direction along the right-hand side increasing in time. Two unstructured meshes are computed, a coarse one with 1793 elements and a finer mesh with 4917 elements. The meshes are shown in Fig.~\ref{fig-Plate}. 
	
Finite element computations are performed using the finite element analysis program (FEAP). Standard hexahedral (8-Nodes) elements are used, they are restricted in the $3^{rd}$ direction employing plane strain conditions, which apply throughout all examples. Stresses are calculated as analytical derivatives of the energies with respect to the strains. Tangent operators are calculated as second derivatives of the energies with respect to the strains, numerically via perturbation. Equilibrium states are obtained by standard Newton-Raphson procedure.

\begin{figure}
	\centering
	\includegraphics[scale=0.54]{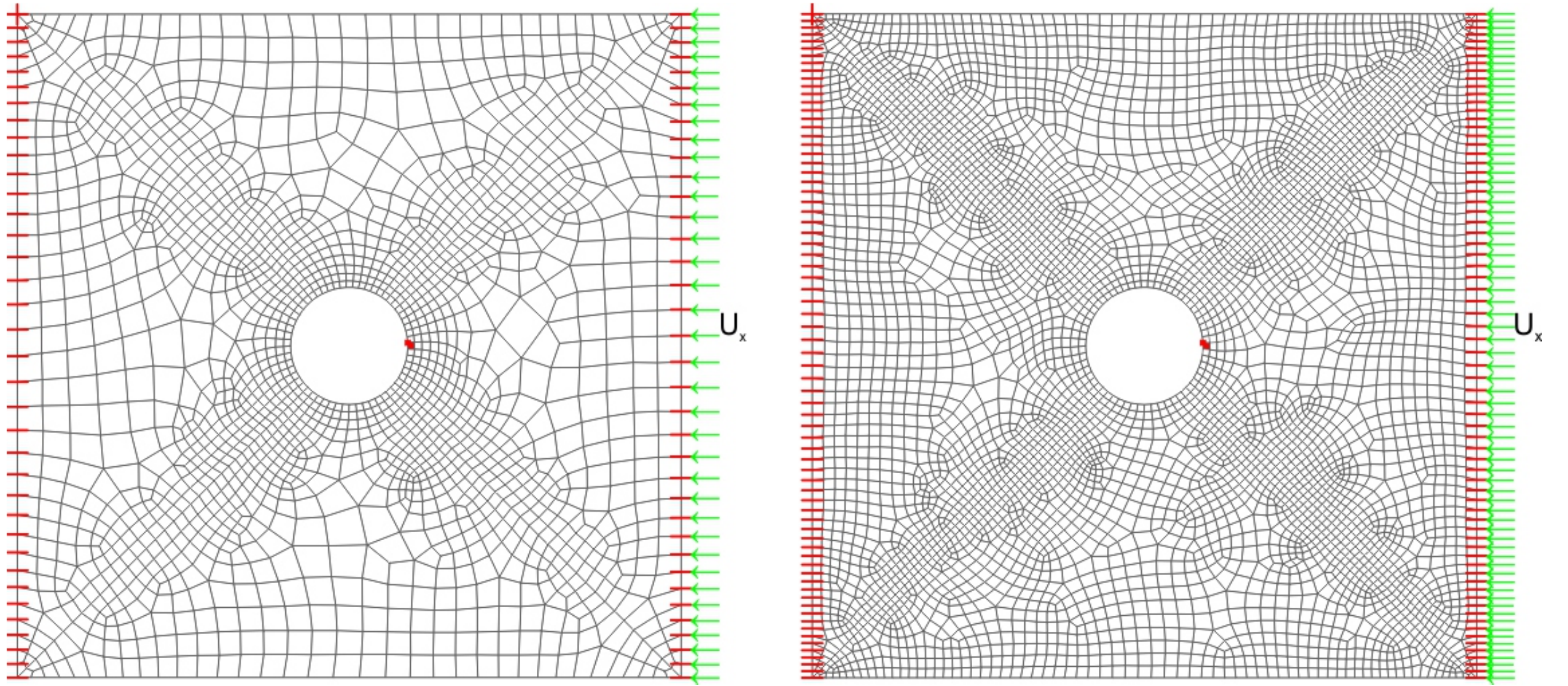}  
	\caption{A plate with a hole is fixed on the left and subject to a given displacement in $x$-direction on the right-hand side. The figure to the left shows the coarse mesh with 1793 elements and that to the right the fine mesh with 4917 elements. The elements employed for the stress plots are marked in red.}
	\label{fig-Plate}
\end{figure}

For the computations, we use the same material parameters as in the one-dimensional case taken from a modified Mohr-Coulomb material model with bulk modulus $K = 3.9 \,\text{GPa}$, shear modulus $\mu = 2.8\, \text{GPa}$, internal friction angle $\phi= 32^\circ$, cohesion $c= 25 \,\text{MPa}$, Poisson's ratio $\nu= 0.25$, and tensile strength $\sigma_\mathrm{t} = 5\, \text{MPa}$, see \cite{Shen2018criticalstate}. Under increasing load, the elements in the neighborhood of the hole start to plastify first. Therefore, we compare the results for the model employing the relaxed energy at the marked elements adjacent to the hole for both meshes. The values are taken for a specific integration point at these elements. We would like to stress, that a comparison with the model employing the condensed energy is not possible, as the lack of convexity causes severe numerical instabilities. Therefore, only the results for the relaxed model can be inferred from Fig.~\ref{fig-OutPlate}. 

\begin{figure}
	\centering
	\includegraphics[scale=0.76]{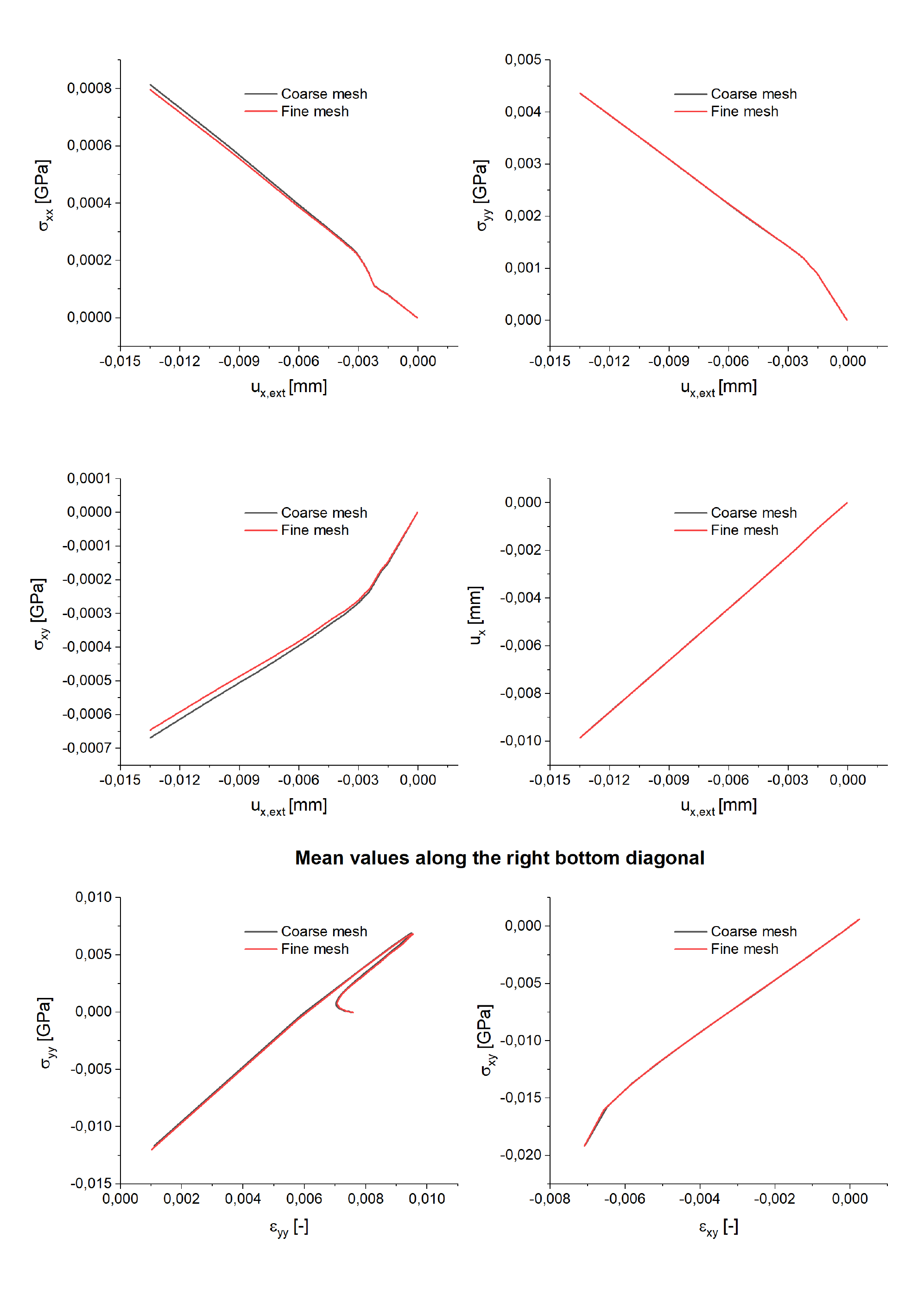}  
	\caption{Square plate with circular hole: from top left to bottom right: stresses $\sigma_{xx}$,  $\sigma_{yy}$, $\sigma_{xy}$ and displacement $u_x$ at the marked elements as a function of the external boundary displacement $u_{x,\mathrm{ext}}$, path of average stress $\sigma_{yy}$ versus average strain $\varepsilon_{yy}$, path of average stress $\sigma_{xy}$ versus average strain $\varepsilon_{xy}$}
	\label{fig-OutPlate}
\end{figure}

The final external displacement is reached employing 100 load steps. The corresponding strains switch
from domain ${Y_1}$ into domain ${Y_2}$, inducing an initiation of microstructure. It can be seen that the behavior of the displacement field and the stresses are almost identical for both the coarse and fine mesh, even after plastification. This demonstrates, that the relaxed model is mesh independent and capable of capturing the material behavior in different zones of relaxation.

The distribution of the stresses $\sigma_{xx}$,  $\sigma_{yy}$, $\sigma_{xy}$ can be seen in Fig.~\ref{fig-Contour}.  The relaxed model shows similar stress distribution for the coarse and fine mesh. Microstructure is initiating at the same positions and load step despite the different spatial discretization.

\begin{figure}
	\centering
	\includegraphics[scale=0.75]{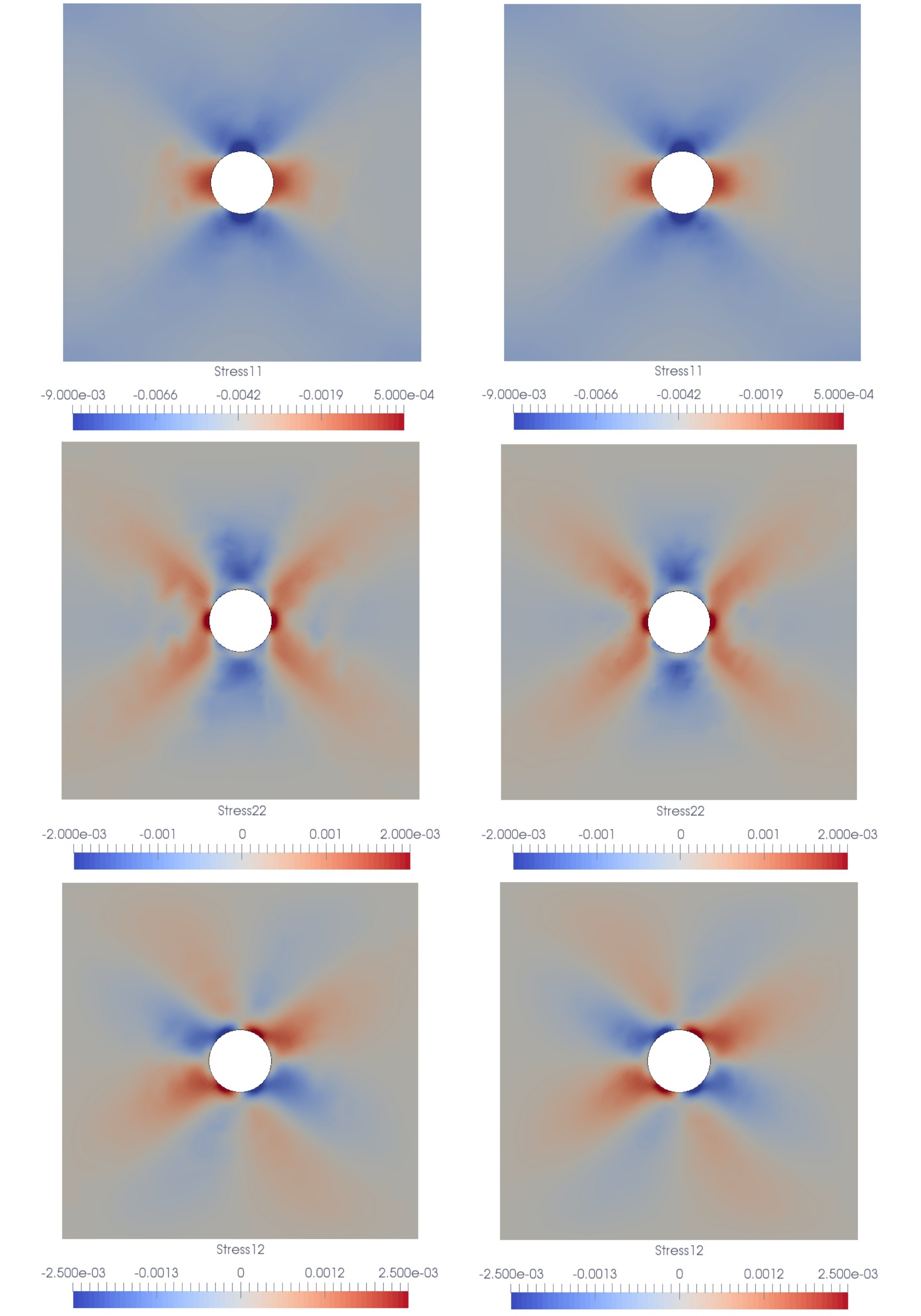}  
	\caption{Contours from the plate with a hole. Distribution of the stresses $\sigma_{xx}$,  $\sigma_{yy}$, $\sigma_{xy}$ is shown. Coarse mesh (left) and fine mesh (right).}
	\label{fig-Contour}
\end{figure}

\section{Conclusion and outlook}
\label{sec8}

We established a variational model for pressure-dependent plasticity as common in models of soil mechanics. Our approach employed a novel aspect by using a characteristic function as dissipation potential. That way, it was possible to establish a dissipation distance as well as a condensed energy connected to a time-incremental setting in a consistent manner. This allowed us to investigate the initiation of microstructure via relaxation theory by calculating the quasiconvex envelope of the condensed energy. For a one-dimensional model problem, for which the quasiconvex envelope coincides with the convex envelope, we succeeded in arriving at a closed-form expression. Interestingly, the quasiconvex envelope turned out to be largely independent of the yield surface of the original model for small values of the hardening parameter $b$. Numerical simulations confirmed the predictions of the analytical results and illustrated the qualitatively different behaviour of the variational models using the condensed and the relaxed energy, respectively. In particular the simulations employing the condensed energy succeeded in finding the necessary oscillations. 
It was possible to extend the one-dimensional formulation to higher dimensions in an empirical manner. Numerical simulations for a paradigmatic model problem provided strong evidence for the superior predictive power of the relaxed model.

However, problems left for future investigations are manifold. The characterization of the relaxed energy for large values of $b$ is a challenging task since the affine curves in Figure~\ref{fig:constructionrelaxation} bounding the domain with $f=f_\mathrm{c}$ intersect the graphs of the functions $\pm r$ and become nonlinear. It is expected that the qualitative behaviour for the relaxation involving at most three points in the phase diagram remains unchanged for an intermediate range of values for $b$ and that for sufficiently large values of $b$ two points will suffice.

From the mathematical side, the treatment of the higher-dimensional case is completely open. In fact, our heuristic model provides a guaranteed lower bound for the energy in the system but may, in general, fail to be the largest lower bound and thus may not be optimal. Concerning variational modeling, there are still problems to be solved connected to the non-associated character of pressure-dependent plasticity. Finally, the treatment of several successive time-increments is still open as well as establishing fully evolutionary models.

\bigskip

\textbf{Acknowledgment:} 
The authors gratefully acknowledge the funding by the German Research Foundation (DFG) within the Priority Program 2256 ``Variational Methods for Predicting Complex Phenomena in Engineering Structures and Materials'' within project 441211072/441468770 ``Variational modeling of pressure-dependent plasticity - a paradigm for model reduction via relaxation”.

\bigskip

\begin{appendix}

\section{Convex analysis}
The construction of $f_\mathrm{c}$ in Theorem~\ref{thm:RelaxationSmallb} was based on the construction of a candidate $f_\mathrm{a}$ and a decomposition of $\R^2$ into disjoint regions with the property that $f_\mathrm{a}$ was globally continuous and convex on the interior of the different regions. This appendix collects the results that are needed in order to prove that $f_\mathrm{a}$ is convex on $\R^2$. The proofs follow from classical results in convex analysis, see~\cite{Behr2023} for more information.

\begin{definition}
Suppose that $\Omega\subset \R^d$ is open and that $f\colon \Omega \to \R$. Then $f$ is said to be locally convex, if for all $x\in \Omega$ there exists an $r_x>0$ with $B(x,r_x)\subset \Omega$ and $f$ convex on $B(x,r_x)$. A vector $v\in \R^d$ is called a local subgradient for $f$ at $x\in \Omega$ if there exists an $r_x>0$ with $B(x, r_x)\subset\Omega$ and if for all $y\in B(x, r_x)$ the inequality $f(y) \geq f(x) + \scp{v}{y-x}$ holds. The set of all local subgradients is denoted by $\partial_\mathrm{loc}f(x)$ and referred to as the local subdifferential.
\end{definition}

The subdifferential in the sense of convex analysis is denoted by $\partial f(x)$. Then, by definition, $\partial f(x)\subset \partial_\mathrm{loc}f(x)$ and if $f$ is differentiable, then $\partial_\mathrm{loc}f(x) \subset \{ f'(x) \}$; however, the local subdifferential may be empty if $f$ the graph of $f$ does not lie above the tangent plane in a sufficiently small neighbourhood of $x$.

The construction of the convex envelope of the condensed energy relies on a local construction and the verification, that the local construction leads to a convex function. The proof uses the following results.

\begin{lemma}\label{lem:C1TranstionsConvex}
Fix $a$, $b\in \R$, $a<0<b$, and $f_a\in C^1([a,0];\R)$, $f_b\in C^1([0,b];\R)$ convex with $f_a(0)=f_b(0)$, $f_a'(0)=f_b'(0)$. Then $f=f_a \chi_{[a,0)} + f_b\chi_{[0,b]}\in C^1([a,b];\R)$ is convex. 
\end{lemma}

\begin{lemma}\label{lem:ConvexUnionSets}
Suppose that $f\in C(\R^d)$, $B$, $C\subset \R^d$ open and disjoint, $C$ convex and  $f$ locally convex on $B\cup C$, $f\in C^1(\interior(B\cup \overline{C}))$. Then $f$ is locally convex on $\interior(B\cup \overline{C})$.
\end{lemma}

\begin{lemma}\label{lem:LocallyConvexLocalSubgradient}
Suppose that $d\geq 2$, $\Omega\subset \R^d$ is open and $f:\Omega \to \R$ is continuous. The following statements are equivalent:
\begin{itemize}
 \item [(i)] $f$ is locally convex in $x\in \Omega$;
 \item [(ii)] for $x\in \Omega$ there exists an $r_x>0$ with $B(x, r_x)\subset\Omega$ such that $\partial_\mathrm{loc}f(y)\neq \emptyset$ for all $y\in B(x, r_x)$.
\end{itemize}

\end{lemma}

\begin{lemma}\label{lem:GlueingConvexFunctions}
Suppose that $a_i$, $b_i\in \R$, $i=1,2$ with $a_1<b_1<a_2<b_2$ and that $f\in C^0([a_1, b_2])$ with $f|_{[a_1,a_2]}$, $f|_{[b_1,b_2]}$ convex. Then $f$ is convex. 
\end{lemma}

\begin{lemma}\label{lem:LocallyConvexIsGloballyConvex}
Suppose that $\Omega\subset\R^n$ is open and that $f\in C(\Omega)$ is locally convex. Then for any pair $a$, $b\in \Omega$ with $\{\lambda a + (1-\lambda)b\colon \lambda\in [0,1]\}\subset \Omega$, the convexity inequality holds on $[a,b]$. In particular, if $\Omega$ is convex, then $f$ is convex on $\Omega$. 
\end{lemma}

\begin{lemma}\label{lem:ConvexityOnConvexSubsetSufficient}
Suppose $\Omega\subset \R^d$ is convex and $f\in C(\R^d)$. If $f\big{|}_\Omega$ is convex and if there is a convex function $g:\R^d\to \R$ with $g\leq f$ on $\R^d$ and $g= f$ on $\R^d\setminus \Omega$, then $f$ is convex.
\end{lemma}

\end{appendix}

\bibliographystyle{acm}

\bibliography{biblio1}

\end{document}